\documentclass[10pt,leqno]{article}
\usepackage{amsfonts}
\usepackage{amsmath, amsthm, amsfonts, amssymb,mathrsfs,color}
\usepackage{indentfirst}
\usepackage{CJK}
\usepackage{tensor}
\usepackage{hyperref}
\hypersetup{
colorlinks,
citecolor=blue,
filecolor=black,
linkcolor=blue,
urlcolor=blue
}

\usepackage{url}

\usepackage{soul}
\usepackage{enumitem}
\usepackage[normalem]{ulem}
\usepackage[numbers,sort&compress]{natbib}

\def\R{\mathbb R}
\def\Z{\mathbb Z}
\def\T{\mathbb T}
\def\N{\mathbb N}
\def\E{\mathbb E}
\def\K{\mathbb K}

\def\p{\mathbb P}
\def\P{\mathcal P}
\def\W{\mathcal W}
\def\F{\mathcal F}
\def\S{\mathcal S}
\def\SS{\mathbf S}
\def\D{\mathcal D}
\def\LL{\mathcal L}
\def\U{\mathbb U}
\def\Q{\mathcal Q}
\def\tt{\widetilde}
\def\hh{\widehat}
\def\d{{\,\rm{d}}}
\def\nn{\nabla}
\def\I{{\mathbf I}}
\def\Wlip{W^{1,\infty}}
\def\OP{{\rm OP}}
\def\B{\mathcal B}
\def\A{\mathcal A}
\def\H{\mathcal H}
\def\M{\mathcal M}
\def\TT{\mathcal T}
\def\ol{\overline}
\def\pas{{\mathbb P}\text{-}{\rm a.s.}}

\theoremstyle{plain}
\numberwithin{equation}{section}
\newtheorem{Theorem}{Theorem}[section]
\newtheorem{Proposition}{Proposition}[section]
\newtheorem{Lemma}{Lemma}[section]
\newtheorem{Definition}{Definition}[section]

\newtheorem{MResult}{Main Result}

\newtheorem{Step}{Step}

\newtheorem{Hypothesis}{Hypothesis}

\newtheorem{Assumption}{Hypothesis}
\newenvironment{ManualHypo}[1]{%
\renewcommand\theAssumption{#1}%
\Assumption
}{\endAssumption}

\theoremstyle{definition}
\newtheorem{Remark}{Remark}[section]

\newcommand{\IP}[1]{\left\langle#1\right\rangle }
\newcommand{\bIP}[1]{\big\langle#1\big\rangle }
\newcommand{\norm}[1]{\left\|#1\right\|}

\usepackage{fancyhdr}
\title{{\bf On stochastic Euler-Poincar\'{e} equations driven by pseudo-differential/multiplicative noise}}
\author{{ \textbf{Hao Tang}}\\
\footnotesize{\it Department of Mathematics,
University of Oslo, P.O. Box 1053, Blindern, N-0316 Oslo, Norway}\\
\footnotesize{\it haot@math.uio.no}}

\begin{document}

\maketitle

\begin{abstract} 
In this paper we focus on the stochastic Euler-Poincar\'{e} equations with pseudo-differential/multiplicative noise. We first establish two new cancellation properties on pseudo-differential operators, which play a key role in energy estimate. Then, we obtain results on local solution, blow-up criterion and global existence. The interplay between stability on exiting times and continuous dependence of solution on initial data is also studied for the multiplicative noise case. 
\end{abstract} 

\vskip 1cm

\tableofcontents 

\section{Introduction and main results}

Let $\I$ be the identity operator. 
Consider the following Euler-Poincar\'{e} (\textbf{EP}) equations: 
\begin{equation}\label{EP}
\partial_t m+\left(u\cdot\nabla\right)m+(\nabla u)^T m+({\rm div}u)m=0,\ \ m=(\I-\alpha\Delta)u,
\end{equation}
where $u=(u_j)_{1\leq j\leq d}$ is the velocity, $m=(m_j)_{1\leq j\leq d}$ with $m_j = (\I-\alpha\Delta)u_j(t,x)$ represents the momentum, $\nabla u^T$ denotes the transpose of $\nabla u$ and $\alpha$ corresponds to the square of the length scale. 
The {\bf EP} equations \eqref{EP} were first studied by Holm et al. \cite{Holm-Marsden-Ratiu-1998-AM,Holm-Marsden-Ratiu-1998-PRL} as a higher dimensional Camassa-Holm (CH) system for modeling and analyzing the nonlinear shallow water waves, see also \cite{Holm-etal-2002-Chapter}.
When  $d\geq2$ and  $m>d/2+3$, we refer to  \cite{Chae-Liu-2012-CMP} for the local existence and uniqueness of a strong solution belonging to $H^m$. Blow-up phenomenon for the case $\alpha=0$ was also obtained in \cite{Chae-Liu-2012-CMP}. For the case $\alpha>0$, the blow-up and global existence of the solutions to \eqref{EP} were studied in \cite{Li-Yu-Zhai-2013-ARMA}. 
For convenience, in this paper we assume $\alpha=1$ in \eqref{EP}.

Now we rewrite \eqref{EP} into the following form (see \cite{Chae-Liu-2012-CMP,Yan-Yin-2015-DCDS,Zhao-Yang-Li-2018-JMAA} for the computation):
\begin{align}\label{EP-u}
u_t+(u\cdot \nabla) u+F(u) =0,
\end{align}
where 
\begin{equation}\label{F-EP}
\left\{\begin{aligned}
F(u)&=(\I-\Delta)^{-1} {\rm div}F_{1}(u)+(\I-\Delta)^{-1}F_{2}(u),\\
F_{1}(u)&= \nabla u (\nabla u +\nabla u^{T})-\nabla u^{T}\nabla u
-\nabla u ({\rm div }u)+\frac{1}{2}I_{d\times d}|\nabla u|^{2},\\
F_{2}(u)&=u({\rm div }u)+u\cdot \nabla u^{T}.
\end{aligned} \right.
\end{equation}
In \eqref{F-EP}, $I_{d\times d}$ is the $d\times d$ identity matrix and $f=(\I-\Delta)^{-1}g$ means $f=G*g$ with the Green function $G$ for the Helmholtz operator $\I-\Delta$. 

\subsection{Pseudo-differential/multiplicative noise}


The introduction of stochasticity into fluid
PDEs has received special attention over the past two decades. The
additional stochastic noise can be a way of representing model uncertainty
and turbulence. For instance, phenomena in weather forecast  including cloud
formation are to this day poorly understood and the inclusion of stochastic
noise has become an essential tool for gaining better understanding about
it. In this paper, we are interested in the following stochastic version of \eqref{EP-u}:
\begin{equation}\label{SEP-u}
\d u+[(u\cdot \nabla) u+F(u)]\d t
= \sum_{k=1}^{\infty}\left(\Q_ku\circ\d\tt W_k + h_k(t,u)\d W_k\right),
\end{equation}
where $\{(\tt W_k, W_k)\}_{k\ge1}$ is a family of independent 1-D Brownian motions, $\{\Q_k\}_{k\ge1}$ is a sequence of pseudo-differential operators (see Section \ref{Section-notations} for details), 
$\{h_k(\cdot,\cdot)\}_{k\ge1}$ is a sequence of nonlinear functions, $\circ\d\tt W_k$ is the Stratonovich stochastic differential and ${\rm d} W_k$ is the It\^o stochastic differential. 

We note the form \eqref{SEP-u} extends many recent results.  If $$\Q_k\equiv 0,\ k\ge1,\ \ \sum_{k=1}^{\infty}\| h_k\|^2_{H^s}<\infty \ \text{for all}\ s>0,$$
where $H^s$ is the Sobolev space with regularity index $s$ (see Section \ref{Section-notations} for precise definition), 
then \eqref{SEP-u} in 1-D becomes the stochastic CH equation investigated in \cite{Tang-2018-SIMA,Rohde-Tang-2021-NoDEA,Rohde-Tang-2020-JDDE}:
\begin{align*}
\d u+\left[uu_{x}+(\I-\partial^2_x)^{-1}\partial_x\left(u^2+\frac12 (\partial_xu)^2\right)\right]\d t=\,&\sum_{k=1}^{\infty} h_k (t,u)\d W_k.
\end{align*}
For the closely related stochastic models, we refer to \cite{Chen-Duan-Gao-2021-PhyD,Zhang-2020-SPA,Rohde-Tang-2020-JDDE,Holden-Karlsen-Pang-2021-JDE,Chen-Duan-Gao-2021-CMS,Tang-Yang-2022-AIHP,Zhang-2020-SPA}. 
If 
$$ h_k\equiv 0,\ \ 
\Q_k=(\I-\partial^2_{x})^{-1}\left\{c_k(\I-\partial^2_{x})
+d_k\partial_x(\I-\partial^2_{x})\right\},\ \ k\ge1$$ for some nice functions ${c_k(x),d_k(x)}$, then \eqref{SEP-u} in 1-D reduces to the stochastic Camassa-Holm (CH) equation with transport noise:
\begin{equation*}
\d m+(um_{x}+2u_{x}m)\d t=\sum_{k=1}^{\infty}(c_km+d_k\partial_xm)\circ\d \tt W_k,\ \ \ \ m=u-u_{xx},
\end{equation*}
which has been studied recently in \cite{Albeverio-etal-2021-JDE}. We also refer to \cite{Alonso-Rohde-Tang-2021-JLNS,Holden-Karlsen-Pang-2021-JDE,Holden-Karlsen-Pang-2022-arXiv} for similar models. 

Here we remark that, as a sequence of pseudo-differential operators, $\Q_k$ considerably extends the well-known transport noise. 
For almost all the known results, the transport noise coefficient can be formulated as:
\begin{equation}\label{known transport noise}
\Q_k=c_k\I+\sum_{j=1}^d d_{k,j}\partial_{x_j}
\end{equation}
for some smooth functions $c_k=c_k(x),d_{k,j}=d_{k,j}(x):\R^d\to \R$ ($1\leq j\leq d$, $k\ge1$). We refer to \cite{Flandoli-Gubinelli-Priola-2010-Invention,Crisan-Holm-2018-PhyD,Crisan-Flandoli-Holm-2018-JNS,Alonso-etal-2019-NODEA,Holden-Karlsen-Pang-2021-JDE,Flandoli-Luo-2020-AoP,Flandoli-Luo-2021-PTRF,Lang-Crisan-2022-SPDE,Albeverio-etal-2021-JDE,Alonso-Bethencourt-2020-JLNS,Alonso-Rohde-Tang-2021-JLNS} and the references therein for different examples with transport noise. 

In the following, we consider the case that $\{\Q_k\}_{k\ge1}$ is a sequence of pseudo-differential operators (see Section \ref{Section-notations} for precise definition) generalizing \eqref{known transport noise}. In this case,  even the problem how to close the {\it a prior} estimate in $H^s$ is 
non-trivial. To see this, rewriting \eqref{SEP-u} into It\^o's form (see \eqref{Cauchy problem-Ito}), and then applying the It\^o formula to $\norm{u}^2_{H^s}$, we will be confronted with the following two terms:
$$\sum_{k=1}^{\infty}\IP{\Q_ku,u}_{H^s} \d \tt W_k\ \text{and}\ \sum_{k=1}^{\infty}\left[\IP{\Q_k^2 u,u}_{H^s}+\IP{\Q_ku,\Q_ku}_{H^s}\right]\d t,$$
which are {\it a prior} singular in terms of $H^s$ since derivatives of order more than $s$ is involved. However, to close the {\it a prior} estimate for $\norm{u}^2_{H^s}$, one has to control the above two terms by ${H^s}$-norm of $u$. 
Such estimates are called cancellation of singularities. And the first main result in this work is
\begin{MResult}[see Theorems \ref{Cancel-Bk} and \ref{Cancel-Ak} for the precise statement]
Cancellation of singularities for two classes of pseudo-differential operators $\{\Q_k\}_{k\ge 1}$, i.e.,
\begin{equation}
\sum_{k=1}^\infty\bIP{\mathcal{P}_n\Q_k f, \mathcal{P}_nf}^2_{L^2}\lesssim\norm{f}^4_{H^s},\ \ 
\sum_{k=1}^\infty\Big|\bIP{ \P_n \Q_k^{2}f, \P_n f }_{L^2}
+ \bIP{\P_n \Q_k f, \P_n \Q_k f }_{L^2} \Big|\lesssim \norm{f}_{H^s}^2,\label{cancel tatget}
\end{equation}
where $\{\P_n\}_{n\ge1}\subset\OP\S^s$ is bounded and $f$ is sufficiently regular.
\end{MResult}

Even though there are already many papers on the existence and uniqueness of solutions to different kinds of SPDEs, as far as we know, there is almost no result on SPDEs with coefficients involving pseudo-differential operators 
except  the recent one \cite{Tang-Wang-2022-arXiv}.  
The above result in this work further extends the cancellation properties in \cite{Tang-Wang-2022-arXiv}. 
With the above cancellation properties, we can obtain the second result in the paper. Roughly speaking, we obtain
\begin{MResult} [See Theorem \ref{Thm-EP} for the detailed statement]
Existence, uniqueness, blow-up criterion  and global existence of solutions to the following Cauchy problem of the stochastic {\bf EP} equations \eqref{SEP-u}, 
\begin{equation}\label{Cauchy problem-S}
\d u+[(u\cdot \nabla) u+F(u)]\d t=\,\sum_{k=1}^{\infty}\left({\mathcal Q}_ku\circ \d \tt W_k+h_k(t,u)\d W_k\right),
\ \ u\big|_{t=0}=\,u_{0}, \ \ x\in \R^d \ \text{or}\ \T^d:=(\R/\Z)^d,
\end{equation}
when $\Q_k$ and $h_k$ satisfy certain conditions.  
\end{MResult}

\subsection{Dependence on the initial data}

For SPDEs, noise effect is one of the probabilistically important questions worthwhile to study and many regularization effects have been observed. We refer to \cite{Flandoli-Luo-2020-AoP,Flandoli-Luo-2021-PTRF,Flandoli-Gubinelli-Priola-2010-Invention,Kroker-Rohde-2012-ANM,Rohde-Tang-2021-NoDEA,Tang-Yang-2022-AIHP,Kim-2010-JFA} and the references therein. 
In this paper, we will consider this noise effect on \eqref{Cauchy problem-S} associated with the dependence on the initial data. 

According to Hadamard, the classical notion of well-posedness of an abstract Cauchy problem requires the existence, uniqueness and the continuous dependence of such solution on initial data. For nonlinear stochastic evolution equations, the property of dependence on initial conditions turns out to be a problem much more complicated than linear growth case or deterministic case. This is because the solution may only exist up to some random time interval $[0,\tau)$ and in general we do \textit{not} have estimates on such $\tau$, cf. \cite{GlattHoltz-Vicol-2014-AP,Rohde-Tang-2020-JDDE,Tang-Yang-2022-AIHP}. 
Indeed, as we will see in Definition \ref{pathwise solution definition}, in stochastic case the solution is a pair $(u,\tau)$. Therefore the dependence on initial conditions is more complicated since we will be confronted with $u_0\mapsto(u,\tau)$ rather than 
$u_0\mapsto u$ on some $[0,\tau]$.

In contrast to most of the previous works where the effects of noise are considered in terms of the regularity or uniqueness of solutions, in this work we consider the dependence on initial conditions.
The question whether (and how) noise can affect initial-data dependence becomes interesting by comparing noise and Laplacian: On one hand,
``regularization by noise'' may formally be related to the regularization produced by an additional Laplacian; On the other hand, if one can indeed add a Laplacian to the governing equations in some cases, then the dependence on initial data may be improved. For example, for the deterministic Euler equations, the dependence on initial data cannot be better than continuous \cite{Himonas-Misiolek-2010-CMP}, but for the deterministic Navier-Stokes equations, it is at least Lipschitz, see pp. 79--81 in \cite{Henry-1981-book}. So far we have not been able to identify the effect from ${\mathcal Q}_ku\circ \d \tt W_k$, hence we consider the following case of \eqref{Cauchy problem-S}:
\begin{equation}\label{EP non stable Eq}
{\rm d}u+\left[\left(u\cdot\nabla\right) u+F(u)\right]{\rm d}t
=\sum_{k=1}^\infty h_k(t,u){\rm d}W_k,\quad u\big|_{t=0}=u_{0}.
\end{equation}

The second result in this paper can be roughly stated as follows:

\begin{MResult}[The detailed statement is  in Theorem \ref{SEP non uniform}]
The solution map $u_0\mapsto (u,\tau)$ defined by \eqref{EP non stable Eq} on $\T^d$ is weakly unstable in the sense that:
\begin{itemize}
\item Either the exiting time of solution $u\equiv0$ is not strongly stable $($see Definition {\rm\ref{Stability on exiting time}}$)$;
\item Or the dependence on initial data is not uniformly continuous.
\end{itemize}
\end{MResult}

\subsection{Plan and remark on the paper} 
\begin{itemize}
\item  In Section \ref{Preliminaries}, we introduce notations and precise the definition of solutions.

\item In Section \ref{Sect:cancel}, we establish the cancellation properties \eqref{cancel tatget} for two board classes of $\Q_k$ not far away from
anti-symmetric. Even from the aspect of pure analysis in pseudo-differential  operators, the results are new, as far as we know.
In Theorem \ref{Cancel-Bk}, the operators depend on $x$ with order $\beta\in[0,1]$. In Theorem \ref{Cancel-Ak}, the operators are independent of $x$ with order $\alpha\ge0$.  Section \ref{Sect:cancel} is ended by Remark \ref{Remark-Bk cancel} and Lemma \ref{Bk-Example-Lemma}, where  we provide some explanations and examples regarding the hypotheses on the operators.

\item We prove existence, uniqueness, blow-up criterion  and global existence of solutions to  \eqref{Cauchy problem-S} (cf. Theorem \ref{Thm-EP}) in Section \ref{Exsitence of regular solutions}. It is worthwhile mentioning that the
conditional expectation $\E[\cdot|\F_0]$ will be used to replace the expectation $\E$ in the construction of solutions and hence we do not assume any moment condition on initial data.
It seems that conditional expectation has been rarely used in the literature of SPDEs. Besides, we remark that the proof for Theorem \ref{Thm-EP} does \textit{not} require any compactness on Sobolev embeddings, which is needed in well-known martingale approach (cf. Prokhorov's Theorem and Skorokhod's Theorem). Hence Theorem \ref{Thm-EP}  holds true not only on torus $\T^d$ but also on the whole space $\R^d$.

\item We study the noise effect on the solution map in Section \ref{Non uniform section}. Our main result is stated in Theorem \ref{SEP non uniform}, which tells us that the multiplicative noise (in It\^o's sense) cannot improve both the stability of the exiting time and  the continuity of the dependence on initial data simultaneously. 
Results of this type seem
to experience less attention in the literature of SPDEs.  

\item Some necessary estimates employed in the proofs are formulated and proved in Appendix \ref{Section:Appendix}.
\end{itemize}

\section{Notations and definitions}\label{Preliminaries}

\subsection{Notations}\label{Section-notations} 

To begin with, we list some notations used subsequently. 
Let $\K=\R$ or $\T:=\R/\Z$ and $d,m\in \mathbb N$. For $1\leq p<\infty$, we denote by
$L^p(\K^d;\R^m)$ the standard Lebesgue space of measurable $p$-integrable $\R^m$-valued functions with domain $\K^d$, and we let $L^\infty(\K^d;\R^d)$ be the space of essentially bounded functions. Particularly, $L^2(\K^d;\R^m)$ has an inner product 
$$
\IP{f,g}_{L^2}:=\sum_{i=1}^m\int_{\K^d}f_i\cdot\overline{g}_i\d x,
$$
where $\overline{g}$ denotes the complex conjugate of $g$. If there is no ambiguity, in the following we denote by $\IP{f,g}_{L^2}$ the inner product for both $f,g\in L^2(\K^d;\R^m)$ and $f,g\in L^2(\K^d;\R)$ with
the customary abuse of notation.

Let   ${\rm i}=\sqrt{-1}$ be the imaginary unit. 
The Fourier transform and inverse Fourier transform are defined by
$$(\mathscr Ff)(\xi):=\int_{\mathbb R^d}f(x){\rm e}^{-{\rm i}(x\cdot \xi)}{\rm d}x,\ \ 
(\mathscr F^{-1}f)(x)= \frac{1}{(2\pi)^{d} }\int_{\mathbb R^d}f(\xi){\rm e}^{{\rm i} (x\cdot \xi)}{\rm d}\xi,\ \ x,\xi\in\mathbb R^d.$$ 
On torus, i.e., $x\in\T^d$, the Fourier and inverse Fourier transforms are defined as
$$(\mathscr Ff) (k):=(2\pi)^d\int_{\mathbb T^d}f(x){\rm e}^{-2\pi {\rm i}(x\cdot k)}{\rm d}x,\ \ 
(\mathscr F^{-1}f)(x)=  \sum_{k\in\Z^d}f(k){\rm e}^{2\pi {\rm i} (x\cdot k)},\ \ x\in\mathbb T^d,\ k\in\Z^d.$$
We remark that the factor  $2\pi$ appears here to guarantee the periodicity of $f(x)$, which can be dropped if we take $\T=\R/(2\pi\mathbb Z)$ instead of $\R/\mathbb Z$. 

Recall that $\I$ stands for the identity mapping.
For any $s\in\R$, the operator $\D^s=(\I-\Delta)^{s/2}$
is defined by
$$\D^s:=\mathscr F^{-1}\left((1+|\cdot|^2)^{\frac s 2}\I\right)\mathscr F.$$
For $s\ge0$, $d,m\ge1$, the Sobolev spaces $H^s$ on $\K^d$ with values in $\R^m$ are defined  as 
$$H^s(\mathbb R^d;\R^m):=\ol{C_0^\infty(\mathbb R^d;\R^m)}^{\|\cdot\|_{H^s}},\ \
H^s(\mathbb T^d;\R^m):=\ol{C^\infty(\mathbb T^d;\R^m)}^{\|\cdot\|_{H^s}},\ \
\|f\|_{H^s}:=\sqrt{\IP{f,f}_{H^s}},$$
where
$$
\IP{f,g}_{H^s}:=\sum_{i=1}^m\IP{\D^s f_i,\D^s g_i}_{L^2}.$$
If  $d,m\in\mathbb N$ are fixed in the context, for $s\in\R$, $p\in [1,\infty]$   we will simply write
$$H^s=H^s(\mathbb K^d;\R^m),\ \ W^{1,\infty}=W^{1,\infty}(\mathbb K^d;\R^m), \ \ L^p=L^p(\mathbb K^d;\R^m),\ \ s\in\R,\ \ p\in [1,\infty],$$ 
where $W^{1,\infty}(\mathbb K^d;\R^m)$ is the set of weakly differential functions $f:\mathbb K^d\to\R^m$ such that 
$$\|f\|_{W^{1,\infty}}:= \sum_{j=1}^m\sum_{|\alpha|_1=0,1}\|\partial_x^{\alpha} f_j\|_{L^\infty} <\infty.$$

Let $\N_0^d:=(\N\cup\{0\})^d$. For two multi-indexes $\alpha=(\alpha_1,\cdots,\alpha_d),\, \beta=(\beta_1,\cdots,\beta_d)\in \N_0^d$ 
with $\beta\le\alpha$ (which means $\beta_i\leq \alpha_i$ with $1\le i\le d$),  we define
$$|\alpha|_1:=\sum_{k=1}^d \alpha_k,\ \ \partial_x^\alpha:= \prod_{k=1}^d \partial_{x_k}^{\alpha_k},\ \ \partial_{\xi}^\alpha:= \prod_{k=1}^d \partial_{\xi_k}^{\alpha_k},\ \ 
\left(\begin{matrix}\alpha \\ \beta\end{matrix}\right):=\prod_{i=1}^d\left(\begin{matrix}\alpha_{i} \\ \beta_{i}\end{matrix}\right)
=  \prod_{i=1}^d\frac{\alpha_i!}{\beta_i!\cdot(\alpha_i-\beta_i)!} .$$
Then for any $s\in \R$, we define the symbol class $\SS^s(\R^d\times \R^d)\subset C^\infty (\R^d\times \R^d;\mathbb{C}^{m\times m})$ as 
\begin{equation}\label{Ss define}
\SS^s(\R^d\times \R^d;\mathbb{C}^{m\times m}):=\left\{\mathfrak{p} \ : \
\forall\, \alpha,\beta\in \N_0^d,\ \exists\, C(\alpha,\beta)>0\ \text{such that}\ \frac{\left|\partial_{x}^{\beta} \partial_{\xi}^{\alpha} 
\mathfrak{p} (x, \xi)\right|_{m\times m}}{(1+|\xi|)^{ s-|\alpha|_1}}<\infty\right\}.
\end{equation} 
Here and in the sequel, $|\cdot|_{m\times m}$ and $|\cdot|$ are usual norms in $\mathbb C^{m\times m}$ and $\R^d$, respectively.  It is well-known that $\SS^s(\R^d\times \R^d;\mathbb{C}^{m\times m})$ is a Fr\'{e}chet space equipped with the topology generated by seminorms 
$\{|\cdot|^{\alpha,  \beta;s}_{\R^d}\}_{\alpha,  \beta\in \N_0^d}$, where
$$|\mathfrak{p}|^{\alpha,  \beta;s}_{\R^d}:=\sup_{(x,\xi)\in\R^d\times \R^d}\frac{|\partial_{x}^{\beta} \partial_{\xi}^{\alpha} {\mathfrak{p}}(x, \xi)|_{m\times m}}{(1+|\xi|)^{ s-|\alpha|_1}}.$$

For any $\alpha\in \mathbb N_0^d$, the   partial difference operator $\triangle^\alpha$ for   a function  $g:\mathbb{Z}^{d}\to \mathbb{C}$ is given by 
$$  (\triangle^\alpha g)(k) = \triangle^\alpha_{k}g(k):=\displaystyle\sum_{\gamma\in\N_0^{d}, \gamma\leq \alpha} (-1)^{|\alpha-\gamma|_1}\binom{\alpha}{\gamma}g(k+\gamma),\ \ k\in \mathbb Z^d.$$
Then   the (toroidal) symbol class of order $s$ for $s\in\R$ is  defined as (cf. \cite{Ruzhansky-Turunen-2010-Book}):
\begin{equation} 
\SS^s(\T^d\times \Z^d;\mathbb{C}^{m\times m}):=\left\{\mathfrak{p} \ : \
\begin{aligned}
&\mathfrak{p}(\cdot,k)\in  C^\infty (\T^d;\mathbb{C}^{m\times m}) \ \text{for all}\ k\in\Z^d;\\
\forall\, \alpha,\beta\in \N_0^d,\ &\exists\, C(\alpha,\beta)>0\ \text{such that}\ \frac{\left|\partial_{x}^{\beta}\triangle_{k}^{\alpha} 
\mathfrak{p} (x, k)\right|_{m\times m}}{(1+|k|)^{ s-|\alpha|_1}}<\infty
\end{aligned}
\right\}
\end{equation} 
Again,
this is a Fr\'{e}chet space under the topology generated by seminorms $\{|\cdot|^{\alpha,  \beta;s}_{\T^d}\}_{\alpha,  \beta\in \N_0^d}$ with 
$$|\mathfrak{p}|^{\alpha,\beta;s}_{\T^d}:=\sup_{(x,k)\in \T^d\times \Z^d}  \frac{|\partial_{x}^{\beta}\triangle_{k}^{\alpha}  \mathfrak{p}(x, k)|}{(1+|k|)^{ s-|\alpha|_1}}.$$ 
Then the pseudo-differential operator  with symbol $\mathfrak{p}$ is defined by
\begin{equation}\label{OP define}
\OP(\mathfrak{p}):=\P,\ \ [\P f](x):= 
\left\{
\begin{aligned}
&\frac{1}{(2\pi)^d} \int_{\mathbb R^d} \mathfrak{p}(x,\xi) (\mathscr Ff)(\xi) {\rm e}^{ {\rm i} (x \cdot \xi)} \d \xi,\ \ \text{if}\ \mathfrak{p}(x,\xi)\in \SS^s(\R^d\times \R^d;\mathbb{C}^{m\times m}),\\
&\sum_{k\in \Z^d} \mathfrak{p}(x,k) 
(\mathscr Ff)(k){\rm e}^{2\pi {\rm i}(x\cdot k)},\qquad \quad  \text{if}\ \mathfrak{p}(x,k)\in\SS^s(\T^d\times \Z^d;\mathbb{C}^{m\times m}).
\end{aligned}\right.
\end{equation}
Throughout this paper, all pseudo-differential operators are assumed to be real-valued,  i.e., when $f$ is real,  $[\OP({\mathfrak{p}})f]$ is also real. Equivalently, it is  required that
\begin{equation*}
\mathfrak{p}(x,-\xi)=\overline{ \mathfrak{p}(x,\xi)}\ \ \text{if}\ \
(x,\xi)\in\R^d \times \R^d\ \ \text{and}\ \ \mathfrak{p}(x,-k)=\overline{ \mathfrak{p}(x,k)}\ \  \text{if}\ \ (x,k)\in\T^d  \times \Z^d.
\end{equation*}

When $m=1$, we remark that, $\mathfrak{p}(x,k)\in\SS^s(\T^d\times\Z^d;\mathbb C)$ if and only if there exists $\tt{\mathfrak{p}} \in\SS(\R^d\times \R^d;\mathbb C)$ such that $\tt{\mathfrak{p}}(x,\xi)$ is periodic in $x$ with period 1 for all $\xi\in\R^d$ (hence $x\in\T^d$), $\tt{\mathfrak{p}}\big|_{\T^d\times\Z^d}=\mathfrak{p}(x,k)$ and (see for example \cite[Theorem 4.5.3 and Corollary 4.5.7]{Ruzhansky-Turunen-2010-Book} and \cite[Theorem 7.2.1]{Saranen-Vainikko-2002-Book})
$$|\tt {\mathfrak{p}}|_{\R^d}^{\alpha,\beta;s}\simeq |\mathfrak{p}|_{\T^d}^{\alpha,\beta;s}.$$
Therefore the bounded subset in
$\SS^s(\T^d\times \Z^d;\mathbb{C})$ coincides with the restriction to $\T^d$ of 
bounded subset in $\SS^s(\R^d\times \R^d;\mathbb{C})$.
If $m\ge1$, this also holds true by  considering each element in the matrix.
Therefore, we simplify notations if there is no ambiguity in the context and we will simply write
\begin{equation}\label{Ss unify}
\SS^s:=\left\{\SS^s(\R^d\times \R^d;\mathbb{C}^{m\times m}) \ : \ \mathfrak{p}(x,-\xi)=\overline{ \mathfrak{p}(x,\xi)} \right\}\ \text{or}\ \left\{\SS^s(\T^d\times \Z^d;\mathbb{C}^{m\times m})\ : \ \mathfrak{p}(x,-k)=\overline{ \mathfrak{p}(x,k)}\right\}.
\end{equation}
In th following we will also consider symbols only depending on the frequency variable $\xi$ (if $x\in\R^d$) or $k$ (if $x\in\T^d$). 
To highlight the differences, we let
\begin{equation}
\SS_0:=\left\{\mathfrak{p}\in \SS^s\ : \ 
\begin{aligned}
&\mathfrak{p}(x,\xi)=\mathfrak{p}(\xi),\ \text{if}\ 
(x,\xi)\in\R^d\times \R^d\\
&\mathfrak{p}(x,k)=\mathfrak{p}(k),\ \text{if}\ 
(x,k)\in\T^d\times \Z^d
\end{aligned}
\right\}
\end{equation}
To emphasize the scalar symbols (i.e., $m=1$ in \eqref{Ss define}), as in \eqref{Ss unify}, we simply write
\begin{equation*}
\S^s:=\left\{\SS^s(\R^d\times \R^d;\mathbb{C}) \ : \ \mathfrak{p}(x,-\xi)=\overline{ \mathfrak{p}(x,\xi)} \right\}\ \text{or}\ \left\{\SS^s(\T^d\times \Z^d;\mathbb{C})\ : \ \mathfrak{p}(x,-k)=\overline{ \mathfrak{p}(x,k)}\right\}.
\end{equation*}
Then for 
$s\in\R$, we  recall \eqref{OP define} and define
\begin{align}
\OP\SS^s:=\Big\{\OP({\mathfrak{p} })\ : \ {\mathfrak{p} } \in\SS^s\Big\},\ \ \OP\SS_0^s:= \Big\{\OP({\mathfrak{p} })\ : \ {\mathfrak{p} } \in \SS_0^s\Big\}.\label{OPS Real}
\end{align}
In the same way, $\OP\S^s$ and $\OP\S_0^s$ can be defined as pseudo-differential operators with symbols in $\S^s$ and $\S_0^s$, respectively.

For linear operators $\mathcal{A}$ and $\mathcal{B}$,   $[\mathcal{A},\mathcal{B}]:=\mathcal{A}\mathcal{B}-\mathcal{B}\mathcal{A}$. $\mathcal{A}^*$ denotes the $L^2$-adjoint operator of the linear operator $\mathcal{A}$. Let $\lesssim$ and $\gtrsim$ denote estimates that hold up to some universal \textit{deterministic} constant which may change from line to line. Let $X$ and $Y$ be two Banach spaces. We denote by $\LL(X;Y)$ the class of bounded linear operators from $X$ to $Y$. To conclude this part, 
we recall, 
cf. \cite[Page 53]{Taylor-1981-Book} and \cite[Theorem 3.41]{Abels-2012-Book},
\begin{align}
{\rm OP}: \SS^{s}\to \LL(H^{r+s};H^r) \ \ \text{is continuous},\ \ r,s\in\R.\label{OP continuous}
\end{align}


%
%

\subsection{Definitions}

Although \eqref{OP continuous} means that $\OP\SS^s$ can be measured by $\LL(H^r;H^{r-s})$,  it is also convenient to consider boundedness in the following sense:
\begin{Definition}
Let $s\in\R$. $\{\P_{n}\}_{n\ge1}\subset \OP\SS^s$ is said to be  bounded if $\P_{n}=\OP(\mathfrak{p}_{n})$ and $\{\mathfrak{p}_{n}\}_{n\ge1}\subset \SS^s$ is bounded  in the sense of boundedness in Fr\'{e}chet space $($cf. {\rm \cite{Rudin-1991-book}}$)$. 
\end{Definition}
To avoid any confusion, for two separable Banach spaces $X$ and $Y$,  $\|\cdot\|_{\LL(X;Y)}$ will always be mentioned if boundedness of $\LL(X;Y)$ is considered.

Next, we give the precise definition of the solutions. To this end,  we first rewrite \eqref{Cauchy problem-S}. 
By using the following formula for a semi-martingale $\xi(t)$: 
$$ \xi\circ \d \tt W_k = \xi  \d \tt W_k+ \frac{1}{2} \d \left\langle \xi, \tt W_k\right\rangle , $$ where $\left\langle\cdot,\cdot\right\rangle$ is the quadratic variation,
\eqref{Cauchy problem-S} can be reformulated as
\begin{align}
\d u=\left[-(u\cdot \nabla) u-F(u)+\frac12\sum_{k=1}^{\infty}{\Q}^2_ku\right] \d t+\sum_{k=1}^{\infty}\left({\Q}_ku \d \tt W_k+h_k(t,u)\d W_k\right),\ \ u\big|_{t=0}=\,u_{0}, \ \ x\in \R^d.\label{Cauchy problem-Ito}
\end{align}
Then we will try to find solutions to \eqref{Cauchy problem-Ito} in the following sense:
\begin{Definition}\label{pathwise solution definition}
Let $d\geq 2$ and let $\K=\R$ or $\T$. Let $u_0$ be an $H^s(\K^d;\R^d)$-valued $\mathcal{F}_0$-measurable random variable with $s>\frac{d}{2}+1$. A local pathwise solution to \eqref{Cauchy problem-Ito} is a pair $(u,\tau)$, where \\
{\bf (1)} $\tau$ is a stopping time satisfying $\p(\tau>0)=1$ and
$(u(t))_{t\in[0,\tau)}$ is an $\mathcal{F}_t$-progressively measurable such that 
$$\sup_{t'\in[0,t]}\|u(t')\|_{H^{s}}<\infty,\ t\in[0,\tau)\ \ \pas,$$ and the following equation holds for all $(t,x)\in[0,\infty)\times\K^d$:
\begin{align*} 
u(t)-u_0
+\int_0^{t}
&\left[\left(u\cdot\nabla\right) u+F(u)-\frac12\sum_{k=1}^{\infty}\Q^2_ku \right](t')\d t'\\
&=\int_0^{t\wedge \tau}\sum_{k=1}^{\infty}{\Q}_ku(t') \d \tt W_k(t')+\int_0^{t}\sum_{k=1}^{\infty}h_k(t',u)\d W_k(t'),\ \ t\in[0,\tau)\ \ \pas
\end{align*}
{\bf (2)}  Additionally,  a local solution $(u,\tau^*)$ is called maximal,  if $\tau^*>0$ almost surely and \begin{equation*}
\limsup_{t\to\tau^*}\|u(t)\|_{H^s}=\infty\ a.s.\ on\ \{\tau^*<\infty\}.
\end{equation*}
If $\tau^*=\infty$ almost surely, then such a solution is called global.
\end{Definition}
We also introduce the following notions on the stability of exiting time.
\begin{Definition}[Stability of exiting time]
\label{Stability on exiting time}
Let $d\geq 2$ and let $\K=\R$ or $\T$. Let $u_0$ be an $H^s(\K^d;\R^d)$-valued $\mathcal{F}_0$-measurable random variable with $s>\frac{d}{2}+1$. Assume that $\{u_{0,n}\}$ is an arbitrary sequence of $H^s(\K^d;\R^d)$-valued $\mathcal{F}_0$-measurable random variables. For each $n$, let $u$ and $u_n$ be the unique solutions to \eqref{EP non stable Eq} with initial value $u_0$ and $u_{0,n}$, respectively. For any $R>0$ and $n\in\N$, define the $R$-exiting time as
\begin{align*}
\tau^R_{n}:=\inf\left\{t\geq0: \|u_n(t)\|_{H^s}>R\right\},\ \
\tau^R:=\inf\left\{t\geq0: \|u(t)\|_{H^s}>R\right\},
\end{align*}
where $\inf \varnothing= \infty$.\\
{\bf (1)} Let $R>0$. If $u_{0,n}\rightarrow u_0$ in $H^{s}$ almost surely implies
\begin{align}
\lim_{n\rightarrow\infty}\tau^R_{n}=\tau^R\ \ \pas,\label{hitting time continuity}
\end{align}
then the $R$-exiting time is said to be stable at $u$.\\
{\bf (2)} Let $R>0$. If $u_{0,n}\rightarrow u_0$ in $H^{s'}$ for all $s'<s$ almost surely also implies \eqref{hitting time continuity}, then the $R$-exiting time is said to be strongly stable at $u$.

\end{Definition}

\section{Cancellation of singularities}\label{Sect:cancel}

In this section, we will develop two abstract cancellation properties to achieve \eqref{cancel tatget} (see Theorems \ref{Cancel-Bk} and \ref{Cancel-Ak}). For the well-known transport noise case, we refer to \cite{Albeverio-etal-2021-JDE,Crisan-Flandoli-Holm-2018-JNS,Alonso-etal-2019-NODEA,Alonso-Bethencourt-2020-JLNS,Alonso-Rohde-Tang-2021-JLNS} and the references therein.

Recall that 
$\{\P_{n}\}_{n\ge1}\subset \OP\SS^s$ is said to be bounded if $\P_{n}=\OP(\mathfrak{p}_{n})$ and $\mathfrak{p}_{n}\subset \SS^s$ is bounded. Remember that $\mathcal{P}^*$ is the $L^2$-adjoint operator of the linear operator $\mathcal{P}$, and $\K=\R$ or $\T$. Then we make the following assumptions:
\begin{Hypothesis}\label{H-cancel-Bk}
Let $d,m\ge1$ and $\beta\in [0,1]$. 
Let $b_k={\rm diag}(b_{k,1},\,\cdots,\,b_{k,m})$ with $b_{k,i}\in H^{\infty}(\K^d,\R):=\cap_{s\ge0}H^{s}(\K^d,\R)$ $(1\le i\le m)$ satisfying $\sum_{k=1}^{\infty}\sum_{i=1}^m\|b_{k,i}\|_{H^s}<\infty$  for all $s\ge0$. 
Assume that
\begin{equation*}
\B_k:={\rm diag} (\B_{k,1},\cdots,\B_{k,m}),\ \ k\ge 1, 
\end{equation*}
such that  $\{\B_{k,i}\}_{k\ge1}\subset\OP\S^{\beta}$ $(1\le i\le m)$ are bounded. Besides, we suppose that
the following conditions hold true:
\begin{enumerate}[label={ $\bf H_1(\arabic*)$}]

\item\label{Bk-Assum-H} 
There is a sequence of operators $\{\H_k\}_{k\ge1}$ such that  $\H_k={\rm diag} (\H_{k,1},\cdots,\H_{k,m})$, $\{\H_{k,i}\}_{k\ge1}\subset\OP\S^{0}$ $(1\le i\le m)$ is  bounded, and
\begin{equation*}
\B_{k}^*=\H_{k}-\B_{k},\ \ k\ge1.
\end{equation*}

\item\label{Bk-Assum-[hB]} There are constants $\sigma_0,\,c\ge0$
such that for any $h\in C_0(\R^d,\R)$ if $\K=\R$ or $h\in C(\T^d,\R)$ if $\K=\T$,
\begin{equation*}
\Big\|\big[(h\mathbf{I})\B_{k,i},\B_{k,i}^*(h\mathbf{I})\big]\Big\|_{\LL(L^2;L^2)}\leq c \|h\|_{H^{\sigma}}^2,\ \ 1\le i\le m,\ \  k\ge1,\ \ \sigma>\sigma_0.
\end{equation*}
\end{enumerate}

\end{Hypothesis}

As extensions of \eqref{known transport noise}, where $\{\partial_{x_j}\}_{1\le j\le d}$ are skew-adjoint operators, we assume that $\{\B_k\}_{k\ge 1}$ are not far away from
skew-adjoint operators.  In \cite{Tang-Wang-2022-arXiv}, $b_k$ is assumed to be constant. 
Since $\B_k$ already depends on $x$,
at first glance it might not be necessary to assume that $b_k$ also depends on $x$. 
However,
we prefer to do so not only because 
it can be easily compared to the well-known case \eqref{known transport noise} but also because the extension is non-trivial and there are subtle adjustments necessary to be clarified in the new situation. We refer to Remark \ref{Remark-Bk cancel} for more details explaining the conditions in \ref{H-cancel-Bk}.

For clarity, in the following we denote by $b_k\B_k:=(b_k\I)\B_k$ and  $\|b_k\|_{H^s}:=\sum_{i=1}^m\|b_{k,i}\|^2_{H^s}$, $s\ge0$. The cancellation properties \eqref{cancel tatget} for $b_k\B_k$ is stated in the following

\begin{Theorem} \label{Cancel-Bk}
Let $s\ge 0$ and $\{\P_n\}_{n\ge1}\subset\OP\S^s$ be bounded. 
We let 
\begin{equation}\label{s0 bk}
s_0:=
1+\frac{d}{2}+\varepsilon_0\ (\forall \, \varepsilon_0>0) \ \ \text{if}\ s\leq\frac{d}{2}+1, \ \text{and}\ 
s_0:=s,\ \ \text{if}\ s>\frac{d}{2}+1.
\end{equation}
Then we have the following assertions:\\
{\bf (1)} If Hypothesis  \ref{H-cancel-Bk} \textbf{without} \ref{Bk-Assum-[hB]} holds true, then there is a constant $C>0$ independent of $n$ such that 
\begin{equation}
\sum_{k=1}^\infty\bIP{\mathcal{P}_n(b_k\B_k) f, \mathcal{P}_nf}^2_{L^2}\leq C\sum_{k=1}^{\infty}\|b_{k}\|^2_{H^{s_0}}\norm{f}^4_{H^s},\ \ f\in H^{s+\beta}(\K^d;\R^m).\label{cancellation-PB}
\end{equation}
{\bf(2)}  Let Hypothesis \ref{H-cancel-Bk} hold and $s\ge1-\beta$. 
There is a constant $C>0$, independent of $k$ and $n$, such that
\begin{align}
\sum_{k=1}^\infty\left|\bIP{ \P_n (b_k\B_{k})^{2}f, \P_n f }_{L^2}
+ \bIP{ \P_n (b_k\B_{k}) f, \P_n (b_k\B_{k})f }_{L^2} \right|\leq C\sum_{k=1}^{\infty}B_k\norm{f}_{H^s}^2,\ \ f\in H^{s+2\beta}(\R^d;\R^m),\label{cancellation-PB2}
\end{align} 
where $\sigma_0$ is given in \ref{Bk-Assum-[hB]} and $B_k:=\norm{b_{k}}_{H^{s_0}}+\norm{b_{k}}^2_{H^{s_0\vee (\sigma_0+1)}}.$ 
\end{Theorem}
\begin{proof}
{\bf (1)}
One first 
infers from \eqref{OP continuous} and \ref{Bk-Assum-H} that
\begin{equation}\label{Pn Qj Tj operator norm}
\sup_{n\ge1}\norm{\P_n }_{\LL(H^{r+s};H^r)},\ \
\sup_{k\ge1}\norm{\mathcal{B}_k}_{\LL(H^{r+\beta};H^r)},\ \ 
\sup_{k\ge1}\norm{\mathcal{H}_k}_{\LL(H^r;H^r)}<\infty,\ \ r\in\R.
\end{equation}
Since $b_k\B_k f=(b_{k,i}\B_{k,i}f_i)_{1\le i\le m}$, we find
\begin{align*}
\bIP{ \P_n(b_k\B_k) f, \P_nf }_{L^2}
=\,& I_{1,k}+ I_{2,k},\\ 
I_{1,k}:=\,&\sum_{i=1}^m\bIP{ [\P_n,b_{k,i}\I]\B_{k,i} f_i, \P_nf_i }_{L^2},\\ 
I_{2,k}:=\,&\sum_{i=1}^m\bIP{ b_{k,i}\P_n\B_{k,i} f_i, \P_nf_i }_{L^2}.
\end{align*}
Since $s_0-s\ge0$ and $\beta\leq 1$, we can infer from Lemma \ref{commutator:Taylor 2-n} (with $q=0$, $\sigma=s_0$ and $r=s$) and \eqref{Pn Qj Tj operator norm} that
$$I_{1,k}\lesssim\norm{b_{k}}_{H^{s_0}}\norm{\B_k f}_{H^{s-1}}\norm{f}_{H^s}\lesssim \|b_{k}\|_{H^{s_0}}\norm{f}^2_{H^s}.$$
Now we estimate $I_{2,k}$. Via  \ref{Bk-Assum-H},
one can observe that
\begin{align*}
I_{2,k}
=\,\sum_{i=1}^m\Big(&\bIP{ [\P_n,\B_{k,i}] f_i,b_{k,i}\P_nf_i }_{L^2}+
\bIP{ \B_{k,i} \P_n f_i,b_{k,i}\P_nf_i }_{L^2}\Big)\\
=\,\sum_{i=1}^m\Big(&\bIP{ [\P_n,\B_{k,i}] f_i,b_{k}\P_nf_i }_{L^2}+
\bIP{ \P_n f_i, -\B_{k,i} (b_{k,i}\P_nf_i) }_{L^2}
+\bIP{ \P_n f_i, \mathcal{H}_{k,i} (b_{k,i}\P_nf_i) }_{L^2}\Big)\\
=\,\sum_{i=1}^m\Big(&\bIP{ [\P_n,\B_{k,i}] f_i,b_{k,i}\P_nf_i }_{L^2}
-\bIP{ \P_n f_i, [\B_{k,i},b_{k,i}\I]\P_nf_i }_{L^2}\\
&-\bIP{ \P_n f_i, b_{k,i}\B_{k,i}\P_nf_i }_{L^2}
+\bIP{ \P_n f_i, \mathcal{H}_{k,i} (b_{k,i}\P_nf_i) }_{L^2}\Big).
\end{align*}
Since all the functions and operators are real, $\bIP{ \P_nf_i,b_{k,i}\B_{k,i}\P_nf_i }_{L^2}=\bIP{ \B_{k,i}\P_nf_i,b_{k,i}\P_nf_i }_{L^2}$, which brings us
\begin{align*}
I_{2,k}
=\,\sum_{i=1}^m\Big(&\bIP{ [\P_n,\B_{k,i}] f_i,b_{k,i}\P_nf_i }_{L^2}
-\bIP{ \P_n f_i, [\B_{k,i},b_{k,i}\I]\P_nf_i }_{L^2}\\
&-\bIP{ \B_{k,i}\P_nf_i,b_{k,i}\P_nf_i }_{L^2}
+\bIP{ \P_n f_i, \mathcal{H}_{k,i} (b_{k,i}\P_nf_i) }_{L^2}\Big)\\
=\,\sum_{i=1}^m\Big(&2\bIP{ [\P_n,\B_{k,i}] f_i,b_{k,i}\P_nf_i }_{L^2}
-\bIP{ \P_n f_i, [\B_{k,i},b_{k,i}\I]\P_nf_i }_{L^2}
+\bIP{ \P_n f_i, \mathcal{H}_{k,i} (b_{k,i}\P_nf_i) }_{L^2}\Big)-I_{2,k}.
\end{align*}
Hence
\begin{align*}
I_{2,k}= \sum_{i=1}^m\Big(&\bIP{ [\P_n,\B_{k,i}] f_i,b_{k,i}\P_nf_i }_{L^2}
-\frac12\bIP{ \P_n f_i, [\B_{k,i},b_{k,i}\I]\P_nf_i }_{L^2}
+\frac12\bIP{ \P_n f_i, \mathcal{H}_{k,i} (b_{k,i}\P_nf_i) }_{L^2}\Big).
\end{align*}
On account of \eqref{Pn Qj Tj operator norm}, $H^{s_0}\hookrightarrow W^{1,\infty}$, Lemmas \ref{Lemma:(pn qm) [P Q]} and \ref{commutator:Taylor 2-n} (with $q=0$, $\sigma=s_0$ and $r=\beta$), we have
$$\left|\sum_{i=1}^m\bIP{ [\P_n,\B_{k,i}] f_i,b_{k,i}\P_nf_i }_{L^2}\right|\lesssim \norm{f}_{H^{s+\beta-1}}\norm{b_{k}}_{L^\infty}\norm{f}_{H^s}
\lesssim\norm{b_{k}}_{H^{s_0}}\norm{f}^2_{H^s},$$
$$\left|\sum_{i=1}^m\bIP{ \P_n f_i, [\B_{k,i},b_{k,i}\I]\P_nf_i }_{L^2}\right|
\lesssim \norm{f}_{H^s} \norm{b_{k}}_{H^{s_0}} \norm{\P_nf}_{H^{\beta-1}}\lesssim\norm{b_{k}}_{H^{s_0}}\norm{f}^2_{H^s},$$
and
$$\sum_{i=1}^m\bIP{ \P_n f_i, \mathcal{H}_{k,i} (b_{k,i}\P_nf_i) }_{L^2}\lesssim \norm{b_{k}}_{H^{s_0}}\norm{f}^2_{H^s}.$$
In conclusion we derive
\begin{align*}
|I_{2,k}|\lesssim \|b_{k}\|_{H^{s_0}}\norm{f}^2_{H^s}.
\end{align*}
Combining the estimates for $I_{i,k}$ with $i=1,2$ and then taking summation $\sum_{k\ge1}$, we obtain \eqref{cancellation-PB}.

\medskip

{\bf (2)} The proof for \eqref{cancellation-PB2} includes the following steps.

\begin{Step}\label{Step:reformulation} 
For $k,n\in\N$ and $1\le i\le m$, we let
\begin{align*}
\begin{cases}
\mathcal{Z}_{k,i}= [b_{k,i}\I,\B_{k,i}]+\H_{k,i}(b_{k,i}\mathbf{I}),\ \ &\mathcal{R}^{(i)}_{1,k}=[b_{k,i}\B_{k,i},\mathcal{Z}_{k,i}],\\ 
\mathcal{R}^{(i)}_{2,k,n}=\big[\P_n ,b_{k,i}\B_{k,i}\big],\ \ &\mathcal{R}^{(i)}_{3,k,n}=\big[\mathcal{R}^{(i)}_{2,k,n},b_{k,i}\B_{k,i}\big],
\end{cases}
\end{align*}
where $\H_k$ is given in \ref{Bk-Assum-H}. We claim that
\begin{align}
\bIP{ \P_n (b_k\B_k)^{2}f, \P_n f }_{L^2} 
+ \bIP{ \P_n (b_k\B_k) f, \P_n (b_k\B_k) f }_{L^2} 
=\, \sum_{j=1}^{6}\sum_{i=1}^m N_{j,i},\label{cancellation-identiy}
\end{align}
where
\begin{equation*}
\begin{cases}
N_{1,i}:=\, \bIP{ \mathcal{R}^{(i)}_{3,k,n}f_i, \P_n f_i }_{L^2},\ \ 
&N_{2,i}:=\, \bIP{ \mathcal{R}^{(i)}_{2,k,n}f_i, \mathcal{R}^{(i)}_{2,k,n} f_i }_{L^2} \\
N_{3,i}=\, \bIP{ \P_n f_i, \mathcal{Z}_{k,i}\mathcal{R}^{(i)}_{2,k,n}f_i }_{L^2},  \ \
&N_{4,i}:=\,-\frac{1}{2}\bIP{ \P_n f_i, \mathcal{R}^{(i)}_{1,k} \P_n f_i }_{L^2} \\
N_{5,i}:=\,\frac{1}{2}\bIP{ \P_n f_i, \mathcal{Z}_{k,i}^{2} \P_n f_i }_{L^2},  \ \
&N_{6,i}:=\,\bIP{ \mathcal{R}^{(i)}_{2,k,n}f_i, \mathcal{Z}_{k,i}\P_n f_i }_{L^2} 
\end{cases} 
\end{equation*}
\end{Step}
To simplify notation, we let 
$$\TT_k:={\rm diag}(\TT_{k,1},\,\cdots,\,\TT_{k,m}),\ \  \TT_{k,i}:=b_{k,i}\B_{k,i}=(b_{k,i}\I)\B_{k,i},\ \ 1\le i\le m.$$ 
By \ref{Bk-Assum-H}, one can immediately find that 
$\TT_{k,i}^*=-\TT_{k,i}+\mathcal{Z}_{k,i}.$
Therefore we arrive at
\begin{align*}
\bIP{ \P_n \TT_{k,i}^{2} f_i, \P_n f_i }_{L^2} 
=\,&\bIP{ \big(\TT_{k,i}\P_n +\mathcal{R}^{(i)}_{2,k,n}\big)\TT_{k,i} f_i, \P_n f_i }_{L^{2}}\notag\\
=\,&\bIP{ \P_n \TT_{k,i} f_i, \TT_{k,i}^* \P_n f_i }_{L^{2}}
+\bIP{ \mathcal{R}^{(i)}_{2,k,n} \TT_{k,i} f_i, \P_n f_i }_{L^{2}}\notag\\
=\,&-\bIP{ \P_n \TT_{k,i} f_i, \TT_{k,i} \P_n f_i }_{L^{2}}+ \bIP{ \P_n \TT_{k,i} f_i, \mathcal{Z}_{k,i} \P_n f_i }_{L^{2}}+ \bIP{ \mathcal{R}^{(i)}_{2,k,n} \TT_{k,i} f_i, \P_n f_i }_{L^{2}}\notag\\
=\,&- \bIP{ \P_n \TT_{k,i} f_i, \P_n \TT_{k,i} f_i }_{L^{2}}
+ \bIP{ \P_n \TT_{k,i} f_i, \mathcal{R}^{(i)}_{2,k,n} f_i }_{L^{2}}\notag\\
&+ \bIP{ \P_n \TT_{k,i} f_i, \mathcal{Z}_{k,i} \P_n f_i }_{L^{2}}+ \bIP{ \mathcal{R}^{(i)}_{2,k,n} \TT_{k,i} f_i, \P_n f_i }_{L^{2}},
\end{align*}
which means
\begin{align*}
&\bIP{ \P_n \TT_{k,i}^{2} f_i, \P_n f_i }_{L^2}
+ \bIP{ \P_n \TT_{k,i} f_i, \P_n \TT_{k,i} f_i }_{L^{2}} \\
=\,&
\bIP{ \P_n \TT_{k,i} f_i, \mathcal{R}^{(i)}_{2,k,n} f_i }_{L^{2}}
+ \bIP{ \P_n \TT_{k,i} f_i, \mathcal{Z}_{k,i} \P_n f_i }_{L^{2}}+ \bIP{ \mathcal{R}^{(i)}_{2,k,n} \TT_{k,i} f_i, \P_n f_i }_{L^{2}}.
\end{align*}
Note that $\P_n $ is of order $s$. Then $\P_n \TT_{k,i}$ is of order $s+\beta\ge s$. Similarly, the order of $\mathcal{R}^{(i)}_{2,k,n} \TT_{k,i}$ may be bigger than $s$. Therefore, by commuting $\P_n $ and $\TT_{k,i}$ and using $\TT^*_{k,i}$ again, we have
\begin{align}
& \bIP{ \P_n \TT_{k,i}^{2} f_i, \P_n f_i }_{L^2}
+ \bIP{ \P_n \TT_{k,i} f_i, \P_n \TT_{k,i} f_i }_{L^{2}}\notag\\ 
=\,& \bIP{ \TT_{k,i} \P_n f_i, \mathcal{R}^{(i)}_{2,k,n} f_i }_{L^{2}}+ \bIP{ \mathcal{R}^{(i)}_{2,k,n} f_i, \mathcal{R}^{(i)}_{2,k,n} f_i }_{L^{2}}
+ \bIP{ \mathcal{R}^{(i)}_{2,k,n} \TT_{k,i} f_i, \P_n f_i }_{L^{2}}+ \bIP{ \P_n \TT_{k,i} f_i, \mathcal{Z}_{k,i}\P_n f_i }_{L^{2}}\notag\\ 
=\,& - \bIP{ \P_n f_i,\TT_{k,i} \mathcal{R}^{(i)}_{2,k,n} f_i }_{L^{2}}+ \bIP{ \P_n f_i, \mathcal{Z}_{k,i} \mathcal{R}^{(i)}_{2,k,n} f_i }_{L^{2}}\notag\\
&+ \bIP{ \mathcal{R}^{(i)}_{2,k,n} f_i, \mathcal{R}^{(i)}_{2,k,n} f_i }_{L^{2}}+ \bIP{ \mathcal{R}^{(i)}_{2,k,n}\TT_{k,i} f_i, \P_n f_i }_{L^{2}} + \bIP{ \P_n\TT_{k,i} f_i, \mathcal{Z}_{k,i}\P_n f_i }_{L^{2}}\notag\\ 
=\,& \bIP{ \mathcal{R}^{(i)}_{3,k,n} f_i, \P_n f_i }_{L^{2}}+ \bIP{ \P_n f_i, \mathcal{Z}_{k,i} \mathcal{R}^{(i)}_{2,k,n} f_i }_{L^{2}}+ \bIP{ \mathcal{R}^{(i)}_{2,k,n} f_i, \mathcal{R}^{(i)}_{2,k,n} f_i }_{L^{2}} + \bIP{ \P_n\TT_{k,i} f_i, \mathcal{Z}_{k,i}\P_n f_i }_{L^{2}},\label{cancellation-rewrite}
\end{align}
where in the last step the fact that all the functions are real-valued is used to obtain 
$$- \bIP{ \P_n f_i, \TT_{k,i} \mathcal{R}^{(i)}_{2,k,n} f_i }_{L^{2}}+ \bIP{ \mathcal{R}^{(i)}_{2,k,n} \TT_{k,i} f_i, \P_n f_i }_{L^{2}}= \bIP{ \mathcal{R}^{(i)}_{3,k,n} f_i, \P_n f_i }_{L^{2}}.$$ 
Note that $\mathcal{Z}_{k,i}$ is self-adjoint. Then,
once again, we find
\begin{align*}
& \bIP{ \P_n\TT_{k,i} f_i, \mathcal{Z}_{k,i}\P_n f_i }_{L^{2}}\\
=\,&
\bIP{ \TT_{k,i}\P_nf_i, \mathcal{Z}_{k,i}\P_n f_i }_{L^{2}}+ \bIP{ \mathcal{R}^{(i)}_{2,k,n}f_i, \mathcal{Z}_{k,i}\P_n f_i }_{L^{2}}\\
=\,& - \bIP{ \P_n f_i, \TT_{k,i} \mathcal{Z}_{k,i} \P_n f_i }_{L^{2}}
+ \bIP{ \P_n f_i, \mathcal{Z}_{k,i}^{2} \P_n f_i }_{L^{2}}
+ \bIP{ \mathcal{R}^{(i)}_{2,k,n} f_i, \mathcal{Z}_{k,i} \P_n f_i }_{L^{2}}\\
=\,& - \bIP{ \P_n f_i, \mathcal{Z}_{k,i}\TT_{k,i} \P_n f_i }_{L^{2}}
- \bIP{ \P_n f_i, \mathcal{R}^{(i)}_{1,k} \P_n f_i }_{L^{2}}+ \bIP{ \P_n f_i, \mathcal{Z}_{k,i}^{2} \P_n f_i }_{L^{2}}
+ \bIP{ \mathcal{R}^{(i)}_{2,k,n} f_i, \mathcal{Z}_{k,i} \P_n f_i }_{L^{2}}\\
=\,& - \bIP{ \mathcal{Z}_{k,i}\P_n f_i, \TT_{k,i} \P_n f_i }_{L^{2}}
- \bIP{ \P_n f_i, \mathcal{R}^{(i)}_{1,k} \P_n f_i }_{L^{2}}+ \bIP{ \P_n f_i, \mathcal{Z}_{k,i}^{2} \P_n f_i }_{L^{2}}
+ \bIP{ \mathcal{R}^{(i)}_{2,k,n} f_i, \mathcal{Z}_{k,i} \P_n f_i }_{L^{2}}.
\end{align*}
Therefore, adding $ \bIP{ \P_n\TT_{k,i} f_i, \mathcal{Z}_{k,i}\P_n f_i }_{L^{2}}$ to both sides of the above equation and then using $\P_n\TT_{k,i}-\TT_{k,i} \P_n =\mathcal{R}^{(i)}_{2,k,n}$ give rise to
\begin{align}
2 \bIP{ \P_n\TT_{k,i} f_i, \mathcal{Z}_{k,i}\P_n f_i }_{L^{2}}=\,& 
- \bIP{ \P_n f_i, \mathcal{R}^{(i)}_{1,k} \P_n f_i }_{L^{2}}+ \bIP{ \P_n f_i, \mathcal{Z}_{k,i}^{2} \P_n f_i }_{L^{2}}
+2 \bIP{ \mathcal{R}^{(i)}_{2,k,n} f_i, \mathcal{Z}_{k,i} \P_n f_i }_{L^{2}}\label{cancellation:P L-k}
\end{align}
Combining \eqref{cancellation-rewrite} and \eqref{cancellation:P L-k} gives
\begin{align*}
& \bIP{ \P_n \TT_{k,i}^{2} f_i, \P_n f_i }_{L^2}
+ \bIP{ \P_n \TT_{k,i} f_i, \P_n \TT_{k,i} f_i }_{L^{2}}\notag\\ 
=\,& \bIP{ \mathcal{R}^{(i)}_{3,k,n} f_i, \P_n f_i }_{L^{2}}+ \bIP{ \P_n f_i, \mathcal{Z}_{k,i} \mathcal{R}^{(i)}_{2,k,n} f_i }_{L^{2}}+ \bIP{ \mathcal{R}^{(i)}_{2,k,n} f_i, \mathcal{R}^{(i)}_{2,k,n} f_i }_{L^{2}}\\
&-\frac12\bIP{ \P_n f_i, \mathcal{R}^{(i)}_{1,k} \P_n f_i }_{L^{2}}
+\frac12 \bIP{ \P_n f_i, \mathcal{Z}_{k,i}^{2} \P_n f_i }_{L^{2}}
+\bIP{ \mathcal{R}^{(i)}_{2,k,n} f_i, \mathcal{Z}_{k,i} \P_n f_i }_{L^{2}}\\
=\,& \sum_{j=1}^6 N_{j,i}.
\end{align*}
Taking summation $\sum_{i=1}^m$ to above identity gives \eqref{cancellation-identiy}.

\begin{Step}\label{Step:uniform bounds}
Now we claim that there is a constant $C>0$ independent of $n$ and $k$ such that for $1\le i\le m$,
\begin{equation}\label{Ek operator norm}
\|\mathcal{Z}_{k,i}\|_{\LL(L^2;L^2)}\leq C \norm{b_{k,i}}_{H^{s_0}},\ \ k\in\N,
\end{equation}
\begin{equation}\label{R1k-n f bound}
\sup_{n\ge1}\|\mathcal{R}^{(i)}_{2,k,n}\|_{\LL(H^{s+\beta-1};L^2)}\leq C\norm{b_{k,i}}_{H^{s_0}},\ \ k\ge1, 
\end{equation}
where $s_0$ is given in \eqref{s0 bk},and
\begin{equation}\label{[R1 Q] operator norm}
\sup_{k,n\ge1}\|[\mathcal{R}^{(i)}_{2,k,n},\B_{k,i}]\|_{\LL(H^{s+2\beta-2};L^2)}<\infty,\ 1\leq i\leq m.
\end{equation}
\end{Step}
For $\mathcal{R}^{(i)}_{2,k,n}$, it holds that
\begin{align*}
\|\mathcal{R}^{(i)}_{2,k,n}g\|_{L^2}= \|[\P_n ,b_{k,i}\B_{k,i}] g\|_{L^2}
\leq\,&
\|[\P_n,b_{k,i}\I]\B_{k,i} g\|_{L^2}
+ \|b_{k,i}[\P_n ,\B_{k,i}] g\|_{L^2}.
\end{align*}
Since the symbols of $\P_n$ is bounded in $\S^s$, by $\beta\le1$, Lemma \ref{commutator:Taylor 2-n} (with $\sigma=s_0$ and $q=0$) and \eqref{Pn Qj Tj operator norm}, we obtain that for sufficiently regular function $g$,
$$\|[\P_n,b_{k,i}\I]\B_{k,i} g\|_{L^2}\leq C \norm{b_{k,i}}_{H^{s_0}}\norm{\B_{k,i} g}_{H^{s-1}}\leq C \norm{b_{k,i}}_{H^{s_0}}\norm{g}_{H^{s+\beta-1}}.$$
Similarly, we infer from \ref{Bk-Assum-H} and Lemma \ref{Lemma:(pn qm) [P Q]} that $[\P_n ,\B_{k,i}]\in \OP\SS^{s+\beta-1}$ and its operator norm is bounded. Therefore, on account of Lemmas \ref{LOP} and \ref{Lemma:(pn qm) [P Q]}, for sufficiently regular function $g$, we have
$$\|b_{k,i}[\P_n ,\B_{k,i}] g\|_{L^2}\leq \norm{b_{k,i}}_{L^{\infty}}\norm{[\P_n ,\B_{k,i}]g}_{L^{2}} \leq C \norm{b_{k,i}}_{H^{s_0}} \norm{g}_{H^{s+\beta-1}}.$$
Collecting the above estimates, we obtain \eqref{R1k-n f bound}.

The estimate \eqref{Ek operator norm} may be obtained in much the same way as \eqref{R1k-n f bound}. Using Lemma \ref{commutator:Taylor 2-n} to $[\B_{k,i},b_{k,i}\I]g$, 
we arrive at
$$
\|[\B_{k,i},b_{k,i}\I] g\|_{L^2} \leq C\|b_{k,i}\|_{H^{s_0}}\|g\|_{H^{\beta-1}}\leq C\|b_{k,i}\|_{H^{s_0}}\|g\|_{L^{2}}.
$$
From this and \eqref{Pn Qj Tj operator norm}, it is easy to see that \eqref{Ek operator norm} holds.

Since
the symbol of $\mathcal{R}^{(i)}_{2,k,n}=[\P_n,b_{k,i}\B_{k,i}]$ 
can be explicitly written down (see for example \cite[(0.3.6)]{Taylor-1991-book} or \cite[Theorem 1.2.16]{Nicola-Rodino-2010-book}) and they form a bounded sequence in $\SS^{s+\beta-1}$. Then, by Lemma \ref{Lemma:(pn qm) [P Q]} once again, we obtain \eqref{[R1 Q] operator norm}.

\begin{Step}\label{Step-finish}
In this step we finish the proof.
\end{Step}
We recall that $s_0>\frac{d}{2}+1$ is given in \eqref{s0 bk}. 
We first note that
\begin{align*}
\sum_{i=1}^m|N_{1,i}| 
\lesssim\, &\sum_{i=1}^m \norm{[\mathcal{R}^{(i)}_{2,k,n},\TT_{k,i}]f_i}_{L^2}\norm{f_i}_{H^s} \\
\lesssim\,&\sum_{i=1}^m \left( \norm{[\mathcal{R}^{(i)}_{2,k,n},b_{k,i}\I]\B_{k,i} f_i }_{L^{2}}+ \norm{b_{k,i}[\mathcal{R}^{(i)}_{2,k,n},\B_{k,i}]f_i }_{L^{2}} \right)\norm{f_i}_{H^s}.
\end{align*}
As in \textbf{Step} \ref{Step:uniform bounds}, the sequence of symbols corresponding to $\{\mathcal{R}^{(i)}_{2,k,n}\}$ is bounded in $\S^{s+\beta-1}$. Applying
Lemma \ref{commutator:Taylor 2-n} to $\big[\mathcal{R}^{(i)}_{2,k,n},b_{k,i}\I\big]\B_{k,i} f_i$ (with $q=0$, $\sigma=s_0$ and $r=s+\beta-1\ge0$) yields
$$ \left\|\big[\mathcal{R}^{(i)}_{2,k,n},b_{k,i}\I\big]\B_{k,i} f_i\right\|_{L^2}\lesssim \|b_{k,i}\|_{H^{s_0}} \|\B_{k,i} f_i\|_{H^{s+\beta-2}}\lesssim \norm{b_{k,i}}_{H^{s_0}}\norm{f_i}_{H^{s}}.$$
Then we can infer from \eqref{[R1 Q] operator norm} and $\beta\leq1$ that
\begin{equation*}
\norm{b_{k,i}[\mathcal{R}^{(i)}_{2,k,n},\B_{k,i}]f_i }_{L^{2}}\leq \norm{b_{k,i}}_{L^{\infty}}\norm{\big[\mathcal{R}^{(i)}_{2,k,n},\B_{k,i}\big]f_i}_{L^{2}} \lesssim \norm{b_{k,i}}_{H^{s_0}} \norm{f_i}_{H^{s}}.
\end{equation*}
To sum up,
$$\sum_{i=1}^m|N_{1,i}| \leq C \norm{b_{k}}_{H^{s_0}}\norm{f}^2_{H^{s}}.$$
For $\sum_{i=1}^mN_{2,i}$, we simply use \eqref{R1k-n f bound} to find that
\begin{align*}
\sum_{i=1}^m|N_{2,i}| = \sum_{i=1}^m\|\mathcal{R}^{(i)}_{2,k,n}f_i\|^2_{L^2}\lesssim \norm{b_{k}}^2_{H^{s_0}}\norm{f}^{2}_{H^s}.
\end{align*}
Then we use \eqref{Pn Qj Tj operator norm}, \eqref{Ek operator norm} and \eqref{R1k-n f bound}  that
\begin{align*}
\sum_{i=1}^m|N_{3,i}|\leq C\sum_{i=1}^m\norm{\P_n f_i}_{L^2} \norm{b_{k,i}}_{H^{s_0}}\norm{\mathcal{R}^{(i)}_{2,k,n}f_i}_{L^2}
\lesssim \sum_{i=1}^m\norm{b_{k,i}}_{H^{s_0}}^2\norm{f_i}^{2}_{H^s}
\lesssim \norm{b_{k}}_{H^{s_0}}^2\norm{f}^{2}_{H^s}.
\end{align*} 
For $N_{4,i}$, we observe that
\begin{align}
[\TT_{k,i},\mathcal{Z}_{k,i}]
=\,&\left[(b_{k,i}\I)\B_{k,i},(b_{k,i}\I)\B_{k,i}-\B_{k,i} (b_{k,i}\I)
+\H_{k,i}(b_{k,i}\mathbf{I})\right]
=\left[(b_{k,i}\mathbf{I})\B_{k,i},\B_{k,i}^*(b_{k,i}\mathbf{I})\right].\label{K4 needs Assum-1}
\end{align}
Therefore \ref{Bk-Assum-[hB]} and \eqref{Pn Qj Tj operator norm} give rise to
\begin{align}
\sum_{i=1}^m|N_{4,i}|
\lesssim \sum_{i=1}^m\|b_{k,i}\|^2_{H^{ \sigma_0+1}}\|\P_nf_i\|^2_{L^{2}}
\lesssim \|b_{k}\|^2_{H^{ \sigma_0+1}}\|f\|^2_{H^{s}},\label{K4 needs Assum-2}
\end{align} 
where $\sigma_0$ is given in \ref{Bk-Assum-[hB]}.
Once again, \eqref{Pn Qj Tj operator norm}, \eqref{Ek operator norm} and \eqref{R1k-n f bound} enable us to derive
\begin{align*} 
\sum_{i=1}^m\left(|N_{5,i}|+|N_{6,i}| \right)
\leq\,& C \|b_{k}\|^2_{H^{s_0}}\norm{f}^{2}_{H^2}.
\end{align*} 
Collecting all these estimates for \eqref{cancellation-identiy}, we see that there is a constant $C>0$ independent of $n$ and $k$ such that
\begin{equation*}
\bIP{ \P_n (b_{k}\B_{k})^{2}f, \P_n f }_{L^2} + \bIP{ \P_n (b_{k}\B_{k}) f, \P_n (b_{k}\B_{k}) f }_{L^2} \leq C \left(\|b_{k}\|_{H^{s_0}}+\|b_{k}\|^2_{H^{s_0}}+\|b_{k}\|^2_{H^{\sigma_0+1}}\right)\norm{f}_{H^s}^2. 
\end{equation*} 
Hence we obtain \eqref{cancellation-PB2}. 
\end{proof}

From the above proof, we see that, if $\P_n\in\OP\S_0^s$, the cancellation properties hold true for another class of operators described by
\begin{Hypothesis}\label{H-cancel-Ak}
Let $d,m\ge1$, $\alpha\ge0$ and $a_k:={\rm diag}(a_{k,1},\,\cdots,\,a_{k,m})$, $\{a_{k,i}\}_{k\ge1}\in l^2$, $(1\le i\le m)$.
We assume
$\A_k:={\rm diag} (\A_{k,1},\cdots,\A_{k,m})$ with $k\ge 1$ 
such that  $\{\A_{k,i}\}_{k\ge1}\subset\OP\S_0^{\alpha}$ $(1\le i\le m)$ are bounded. Besides, suppose that 
\begin{equation*}
\A_{k}^*=\M_{k}-\A_{k},\ \ k\ge1,
\end{equation*} 
for  some $\M_k:={\rm diag} (\M_{k,1},\cdots,\M_{k,m})$ such that $\{\M_{k,i}\}_{k\ge1}\subset\OP\S_0^{0}$ $(1\le i\le m)$ are bounded.
\end{Hypothesis}

Indeed,  if $\P_n\in\OP\S_0^s$, repeating the proof for Theorem \ref{Cancel-Bk} with noting that in this case  $\A_k\in\OP\SS_0^{\alpha}$, we have
$[\P_n,\A_k]=[\P_n,a_k]=[\A_k,a_k]=0$
and hence  
\begin{equation*}
\bIP{\mathcal{P}_n(a_k\A_k) f, \mathcal{P}_nf}^2_{L^2}=\frac{a^2_{k}}{2}\bIP{ \P_n f, \M_k \P_nf }^2_{L^2},
\end{equation*}
and
\begin{equation*}
\Big|\bIP{ \P_n (a_k\A_{k})^{2}f, \P_n f }_{L^2}
+ \bIP{ \P_n (a_k\A_{k}) f, \P_n (a_k\A_{k})f }_{L^2} \Big|\leq  \frac{a_k^2}{2}\left|\IP{ \P_n f, \M_k^{2} \P_n f }_{L^2}\right|.
\end{equation*}
Therefore we have actually established the following cancellation properties for $x$-independence operators:

\begin{Theorem}\label{Cancel-Ak}
Let $s\ge 0$ and $\{\P_n\}_{n\ge1}\subset\OP\S_0^s$ be bounded.
If Hypothesis \ref{H-cancel-Ak} holds true, then there is a constant $C>0$ independent of $n$ such that 
\begin{equation*}
\sum_{k=1}^\infty\bIP{\mathcal{P}_n(a_k\A_k) f, \mathcal{P}_nf}^2_{L^2}\leq C\|a_k\|^2_{l^2}\norm{f}^4_{H^s},\ \ f\in H^{s+\alpha}(\R^d;\R^m),
\end{equation*}
and 
\begin{align*}
\sum_{k=1}^\infty\Big|\bIP{ \P_n (a_k\A_{k})^{2}f, \P_n f }_{L^2}
+ \bIP{ \P_n (a_k\A_{k}) f, \P_n (a_k\A_{k})f }_{L^2} \Big|\leq C \|a_k\|^2_{l^2}\norm{f}_{H^s}^2, \ \ f\in H^{s+2\alpha}(\R^d;\R^m).
\end{align*} 
\end{Theorem}

\begin{Remark}\label{Remark-Bk cancel}
A few comments are in order regarding Hypothesis \ref{H-cancel-Bk} and Theorem \ref{Cancel-Bk}. 

\textbf{(1)}  When we construct approximation scheme on \eqref{Cauchy problem-Ito} (see \eqref{GH-n} below), mollifiers $J_n$ can not commute with $\B_k$. Therefore, to obtain uniform estimate in $H^s$, we have to deal with $\P_n=\D^sJ_n$ (see Lemma \ref{Gn Hn Lemma} below). Therefore we state the uniform (in $n$) estimate for a sequence $\{\P_n\}_{n\ge1}$ rather than just one $\mathcal{P}$.  

\textbf{(2)} Let $b(x)\in H^{\infty}(\K^d;\R)$.
Now we consider the following question: 
Is there  a $q\ge0$ such that for all $\mathcal{P}\in\OP\S^s$ and $\Q\in\OP\S^{1}$,
\begin{equation}
\left\|\Big[\big[\mathcal{P} ,b(x)\Q\big],b(x)\Q\Big]\right\|_{\LL(H^s;L^2)}\lesssim \|b\|^2_{H^q} ?\label{Problem:commutator}
\end{equation}
So far we have only  been able to show (see the estimate for $L_1$ in \textbf{Step} \ref{Step-finish}
in the proof for  \eqref{cancellation-PB2}) that $$\left\|\Big[\big[\mathcal{P} ,b(x)\Q\big],b(x)\Q\Big]\right\|_{\LL(H^s;L^2)}\lesssim \|b\|_{H^{s_0}},$$
which is weaker than \eqref{Problem:commutator} in the case $\|b\|_{H^{s_0}}<1$. And this is why we have to assume $\sum_{k=1}^\infty \|b_k\|_{H^{s}}<\infty$, stronger than $\sum_{k=1}^\infty \|b_k\|^2_{H^{s}}<\infty$.
In \cite{Crisan-Flandoli-Holm-2018-JNS}, $\mathcal{P}_n=\P=\Delta$ and $\B_k$ is of the form \eqref{known transport noise}. Then Leibniz rule holds true. Without Leibniz rule, so far it is not clear how to verify \eqref{Problem:commutator}. 

\textbf{(3)} Now we explain \ref{Bk-Assum-[hB]}. On account of \ref{Bk-Assum-H} and Lemma \ref{LOP}, we see that for each $k\ge1$
\begin{align*}
[(h\mathbf{I})\B_{k,i},\B_{k,i}^*(h\mathbf{I})]
=\,&[(h\mathbf{I})\B_k,h\mathbf{I}\B_k+\B_k^*(h\mathbf{I})]\\
=\,&[(h\mathbf{I})\B_k,h\mathbf{I}\B_k-\B_k (h\mathbf{I})+\H_k(h\mathbf{I})]
=\big[(h\mathbf{I})\B_k,[h\mathbf{I},\B_k]\big]
+[(h\mathbf{I})\B_{k,i},\H_{k,i}(h\mathbf{I})]
\end{align*}
is actually an operator of order $0$ (since $\beta\leq 1$). Even though its operator norm can be independent of $k$ (via Hypothesis \ref{Bk-Assum-H} and Lemma \ref{Lemma:(pn qm) [P Q]}), it 
depends on $h$. Namely, for $k\ge1$, 
there is a constant $C=C(h)>0$ such that 
$$\norm{\left[(h\mathbf{I})\B_{k,i},\B_{k,i}^*(h\mathbf{I})\right]g}_{L^2}\leq C(h)\norm{g}_{L^2}.$$
Then \ref{Bk-Assum-[hB]} actually means that constant $C(h)$ can be replaced by $c\|h\|^2_{H^{\sigma}}$ with $\sigma\ge\sigma_0$ and $c>0$. This enables us to take summation $\sum_{k=1}^\infty$ (see \eqref{K4 needs Assum-1} and \eqref{K4 needs Assum-2} in the proof). Without \ref{Bk-Assum-[hB]}, even if $\sum_{k=1}^\infty \|b_k\|_{H^{s}}<\infty$, one can \textit{only} take summation of finitely many $k$ in \eqref{cancellation-PB2} since \eqref{K4 needs Assum-2} becomes
$\sum_{i=1}^m|N_{4,i}|
\leq C(b_{k})\|f\|^2_{H^{s}}$ for some constant $C(b_{k})>0$
and we do \textbf{not} \textit{a prior} know whether or not $\sum_{k=1}^\infty C(b_{k})<\infty$. For other qualitative research on second-order commutator in stochastic setting, we refer to
\cite{Holden-Karlsen-Pang-2022-arXiv}.

\end{Remark}

A broad class of operators satisfying \ref{Bk-Assum-[hB]}
are given by

\begin{Lemma}\label{Bk-Example-Lemma}
If $\B_k$ satisfies \ref{Bk-Assum-H} with $0\leq \beta\leq\frac12$ $\Longrightarrow$ $\B_k$ satisfies \ref{Bk-Assum-[hB]} with $\sigma_0=\frac{d}{2}+1$.
\end{Lemma}
\begin{proof}
Keep in mind that $1\le i\le m$ in the following. We note that
\begin{align*}
[h\mathbf{I}\B_{k,i},\B_{k,i}^*(h\mathbf{I})]
= -\left[h\mathbf{I}\B_{k,i},\B_{k,i}(h \mathbf{I})\right]
+\left[h\mathbf{I}\B_{k,i},\H_{k,i}(h \mathbf{I})\right]
:=\,&\Theta_1+\Theta_2.
\end{align*}
Direct computation shows that 
$$\Theta_1=
-h[\B_{k,i}^2,h \I]+[\B_{k,i},h \I](h\I\B_{k,i})+h [\B_{k,i},h\I]\B_{k,i}
:=\sum_{j=1}^3\Theta_{1,j},$$
$$\Theta_2=h[\B_{k,i}\H_{k,i},h\I ]+h^2 [\B_{k,i},\H_{k,i}]+
[h^2 \I,\H_{k,i}]\B_{k,i}:=\sum_{i=1}^3\Theta_{2,j}.$$
Again, \eqref{OP continuous} and \ref{Bk-Assum-H} give us
\begin{equation*}
\sup_{k\ge1}\norm{\B_{k,i}}_{\LL(H^{r+\beta};H^r)},\ \ 
\sup_{k\ge1}\norm{\H_{k,i}}_{\LL(H^r;H^r)}<\infty,\ \ r\in\R.
\end{equation*}
Remember the above estimate and let $\eta>\frac{d}{2}+1$.
By Lemmas \ref{LOP} and \ref{commutator:Taylor 2-n} (with $\sigma=\eta>r=2\beta$ and $q=0$) and $H^{\eta}\hookrightarrow W^{1,\infty}$, we see that
\begin{align*}
\|\Theta_{1,1}g\|_{L^2}\lesssim \norm{h}_{L^\infty}\norm{[\B_{k,i}^2,h \I]g}_{L^2}
\lesssim \norm{h}_{H^\eta}^2 \norm{g}_{L^{2}}.
\end{align*}
Similarly, using $H^{\eta}\hookrightarrow W^{1,\infty}$ and Lemma \ref{commutator:Taylor 2-n} (with $\sigma=\eta>r=\beta$ and $q=0$) and Lemma \ref{Lemma:product in Hs} (with $s_1=\beta-1$ and $s_2=\eta$), we arrive at
\begin{align*}
\|\Theta_{1,2}g\|_{L^2}\lesssim \norm{h }_{H^{\eta}}\norm{h \B_{k,i} g}_{H^{\beta-1}}
\lesssim \norm{h }_{H^{\eta}}\norm{h}_{H^{\eta}}\norm{\B_{k,i} g}_{H^{\beta-1}}\lesssim \norm{h}_{H^\eta}^2\norm{g}_{L^{2}},
\end{align*}
and 
\begin{align*}
\|\Theta_{1,3}g\|_{L^2}\lesssim \norm{h }_{L^{\infty}} \norm{h }_{H^{\eta}}\norm{ \B_{k,i} g}_{H^{\beta-1}}
\lesssim \norm{h}_{H^\eta}^2\norm{g}_{L^{2}}.
\end{align*}
Due to \ref{Bk-Assum-H}, as in the proof for Lemma \ref{Lemma:(pn qm) [P Q]}, the symbols of the product operator $\B_{k,i}\H_{k,i}$ is also bounded in $\S^{\beta}$. Hence we apply Lemma \ref{commutator:Taylor 2-n} (with $\sigma=\eta>r=\beta$ and $q=0$) to find
\begin{align*}
\|\Theta_{2,1}g\|_{L^2}\lesssim \norm{h}_{L^{\infty}}\norm{[\B_{k,i}\H_{k,i},h \I]g}_{L^{2}}
\lesssim \norm{h}_{H^{\eta}}\norm{h }_{H^{\eta}}\norm{g}_{H^{\beta-1}}\lesssim \norm{h}_{H^\eta}^2\norm{g}_{L^{2}}.
\end{align*}
Using Lemma \ref{Lemma:(pn qm) [P Q]} to $\Theta_{2,2}$ yields
\begin{align*}
\|\Theta_{2,2}g\|_{L^2}\lesssim \norm{h^2}_{L^{\infty}}\norm{[\B_{k,i},\H_{k,i}]g}_{L^{2}}
\lesssim \norm{h}_{H^{\eta}}^2\norm{g}_{H^{\beta-1}}\lesssim \norm{h}_{H^{\eta}}^2\norm{g}_{L^{2}}.
\end{align*}
Finally, Lemma \ref{commutator:Taylor 2-n} (with $\sigma=\eta>s=0$ and $q=0$) gives rise to
\begin{align*}
\|\Theta_{2,3}g\|_{L^2}\lesssim \norm{h^2}_{H^{\eta}}\norm{\B_{k,i} g}_{H^{-1}}
\lesssim \norm{h}_{H^\eta}^2\norm{g}_{H^{\beta-1}}\lesssim \norm{h}_{H^{\eta}}^2\norm{g}_{L^{2}}.
\end{align*}
Hence \ref{Bk-Assum-[hB]} holds true with $\sigma_0=\frac{d}{2}+1$. 
\end{proof}

\section{Local and global results}\label{Exsitence of regular solutions}
In this section we focus on \eqref{Cauchy problem-S} and we will use Theorems \ref{Cancel-Bk} and \ref{Cancel-Ak} with $m=d$. As before we simply write
$$H^s=H^s(\K^d;\R^d),\ \ \K=\R\ \text{or}\ \T.$$
We recall  the following estimates for  $F(\cdot)$:
\begin{Lemma}[\cite{Yan-Yin-2015-DCDS,Zhao-Yang-Li-2018-JMAA}]\label{F-EP Lemma}
Let $s>d/2$ with $d\geq 2$. 
The non-local term $F(\cdot)$ defined in \eqref{F-EP} satisfies
\begin{align*}
\|F(v)\|_{H^s}&\lesssim \|v\|_{W^{1,\infty}} \|v\|_{H^s},\ \ s>d/2+1,\ v\in H^s,\\
\|F(v_1)-F(v_2)\|_{H^s} &\lesssim\left(\|v_1\|_{H^s}+\|v_2\|_{H^s}\right)\|v_1-v_2\|_{H^s},\ \ s>d/2+1,\ v_1,\, v_2\in H^s\\
\|F(v_1)-F(v_2)\|_{H^s}
&\lesssim\left(\|v_1\|_{H^{s+1}}+\|v_2\|_{H^{s+1}}\right)\|v_1-v_2\|_{H^s},\ \ d/2+1\geq s> d/2, \ v_1,\, v_2\in H^{s+1}.
\end{align*}
\end{Lemma}

Let $J_n$ be the Friedrichs mollifier defined in  Appendix \ref{Section:Appendix} (cf. \eqref{Define Jn}). Then we have
\begin{Lemma}\label{Te uux+F(u)}
For all $\sigma>\frac{d}{2}+1$,
there is a constant $\varLambda=\varLambda(\sigma,d)>0$ such that 
\begin{align}
\left|\bIP{(u\cdot\nabla)u+F(u),u}_{H^{\sigma}}\right|\leq \varLambda \|u\|^2_{H^{\sigma}}\|u\|_{W^{1,\infty}},\ \ u\in H^{\sigma+1},\label{uux F 1}
\end{align}
\begin{align}
\left|\bIP{J_n [(u\cdot\nabla) u]+ J_n F(u), J_nu}_{H^{\sigma}}\right|\leq \varLambda \|u\|^2_{H^{\sigma}}\|u\|_{W^{1,\infty}},\ \ u\in  H^{\sigma}. \label{uux F 2}
\end{align}
\end{Lemma}
\begin{proof}
We only prove \eqref{uux F 2} since \eqref{uux F 1} can be proved in the same way.
Using Lemmas \ref{Lemma-Je}, \ref{Te commutator} and \ref{Kato-Ponce commutator estimate}, integration by parts and $H^s\hookrightarrow W^{1,\infty}$, we obtain that for some $\varLambda=\varLambda(\sigma,d)>0$,
\begin{align*}
&\left(\D^{\sigma}J_n\left[(u\cdot\nabla)u\right],\D^{\sigma}J_n u\right)_{L^2}\notag\\
=\,&\left(\left[\D^{\sigma},(u\cdot\nabla)\right]u,\D^{\sigma}J^2_n u\right)_{L^2}+
\left([J_n,(u\cdot\nabla)]D^{\sigma}u, \D^{\sigma}J_n u\right)_{L^2} +\left((u\cdot\nabla)\D^{\sigma}J_n u, \D^{\sigma}J_n u\right)_{L^2}\notag\\
\leq\, & \varLambda\left(\|u\|_{H^{\sigma}}\|\nabla u\|_{L^\infty}\|J_n u\|_{H^{\sigma}}
+\|u\|_{H^{\sigma}}\|\nabla u\|_{L^\infty}\|J_n u\|_{H^{\sigma}}
+\|J_n u\|^2_{H^{\sigma}}\|\nabla u\|_{L^\infty}\right)\notag\\
\leq\, & \varLambda\|u\|^2_{H^{\sigma}}\|u\|_{W^{1,\infty}}.
\end{align*}
Similarly, Lemma \ref{F-EP Lemma} implies
\begin{align*}
&\left(\D^{\sigma}J_n F(u),\D^{\sigma}J_n u\right)_{L^2}
\leq \varLambda\|u\|^2_{H^{\sigma}}\|u\|_{W^{1,\infty}}.
\end{align*}
Combining the above estimates gives \eqref{uux F 2}.
\end{proof}
To obtain a solution, we need the following 
\begin{Hypothesis}\label{H-Q}
Let $d\ge1$. 
We assume
\begin{equation}\label{define Qk}
\Q_k=a_k \A_k+b_k \B_k,\ \ a_kb_k=0,\ \ k\ge 1, 
\end{equation}
and $(b_k,\B_k)$ and $(a_k,\A_k)$ satisfy Hypotheses \ref{H-cancel-Bk} and \ref{H-cancel-Ak} with $m=d$, respectively.
\end{Hypothesis}

\begin{Hypothesis}\label{H-hk}
For all $k$,
$h_k:[0,\infty)\times H^s\ni (t,u)\mapsto h_k(t,u)\in H^s$ is continuous for $s>\frac{d}{2}+1$. Moreover, there is a function 
$K:[0,\infty)\times[0,\infty)\rightarrow(0,\infty)$ increasing in both variables such that
\begin{align*}
&\sum_{k=1}^{\infty}\|h_k(t,u)\|^2_{H^s}
\leq\, K(t,\|u\|_{\Wlip})(1+\|u\|^2_{H^s}),\ \ t\ge0,\ u\in H^s,\\
\sum_{k=1}^\infty \|h_k(t,&u)- h_k(t, v)\|_{ H^s}^2 
\le\, K(t, \|u\|_{ H^s}+\|v\|_{ H^s})\|u-v\|^2_{ H^s},\ \ t\ge0,\ u,v\in H^s.
\end{align*}
\end{Hypothesis}

Recall $(\beta, b_k)$ and $(\alpha,a_k)$ in Hypotheses \ref{H-cancel-Bk} and \ref{H-cancel-Ak}, respectively. Let
\begin{equation}\label{p0}
p_0:=\max\Big\{
\alpha{\bf 1}_{\{\|a_k\|_{l^2}>0\}},\, \beta{\bf 1}_{\{\sum_{k=1}^\infty\|b_k\|_{H^{s_0}}>0\}}
\Big\}.
\end{equation} 

The main results for \eqref{Cauchy problem-S} (or \eqref{Cauchy problem-Ito}) is the following
\begin{Theorem}\label{Thm-EP}
Let Hypotheses \ref{H-Q} and \ref{H-hk} be verified. Let $s>\frac{d}{2}+1+\max\{2p_0,1\}$ with $d\geq 2$. For any $H^s$-valued $\mathcal{F}_0$-measurable random variable $u_0$, 
\begin{enumerate}[label={ $\bf (\Roman*)$}]
\item \label{T1} \eqref{Cauchy problem-S} admits a unique maximal solution $(u,\tau^*)$ in the sense of Definitions \ref{pathwise solution definition}. Besides, $(u,\tau^*)$ satisfies $\p\big(u\in C([0,\tau^*);H^s)\big)=1$ and  
\begin{equation}\label{blow-up criterion}
\displaystyle{\bf 1}_{\left\{\limsup_{t\rightarrow \tau^*}\|u(t)\|_{H^s}=\infty\right\}}
={\bf 1}_{\left\{\limsup_{t\rightarrow \tau^*}\|u(t)\|_{W^{1,\infty}}=\infty\right\}}
\ \ \pas
\end{equation}

\item \label{T2} $u$ exists globally, i.e., $\p(\tau^*=\infty)=1$, if 
\begin{equation}\label{NE SPDE}
\limsup_{\|f\|_{H^{\eta}} \to \infty}
\frac{\Xi(T,f,\eta)}{2\varLambda\|f\|_{\Wlip}\|f\|_{H^{\eta}}^2} 
<-1,\ \ T\in(0,\infty),\ \ \eta\in\Big(\frac{d}{2}+1,\, s-\max\{2p_0,1\}\Big), 
\end{equation}
where $\varLambda$ is given in Lemma \ref{Te uux+F(u)} and
\begin{equation}\label{NE h-k}
\Xi(T,v,\sigma):=\sup_{t\in [0,T]}\sum_{k=1}^\infty \left(\left\|h_k(t,v)\right\|_{H^{\sigma}}^2 
- \frac{2\bIP{ h_k(t,v),v}_{H^{\sigma}}^2}{{\rm e}+ \|v\|_{H^{\sigma}}^2}\right),\ \ \sigma>\frac{d}{2}+1.
\end{equation}
\end{enumerate}

\end{Theorem}

The proof for Theorem \ref{Thm-EP} can be carried out in a way similar to \cite{Tang-Wang-2022-arXiv}. However, since the pseudo-differential operators in this paper are extended, here we also provide the details  and the proof is divided into three subsections.

\subsection{Approximation scheme and estimates}

For convenience, we recall that \eqref{Cauchy problem-S} is equivalent to \eqref{Cauchy problem-Ito}.
Because $a_kb_k=0$, we have 
$\Q_k^2=(a_k\A_k)^2+(b_k\B_k)^2$ and then we further rewrite \eqref{Cauchy problem-Ito} as

\begin{align}\label{Cauchy problem-Ito-AB}
\left\{\begin{aligned}
\d u=\,&\bigg\{-(u\cdot \nabla) u-F(u)+\frac12\sum_{k=1}^{\infty}\left[\left(a_k \A_k\right)^2 u+\left(b_k \B_k\right)^2 u\right] \bigg\}\d t\\
&+\sum_{k=1}^\infty \bigg\{\left(a_k\A_k \right)u \d \ol W_k 
+\left(b_k\B_k \right)u \d\hh W_k +  h_k(t,X(t))\d W_k(t)\bigg\},\ \  u\big|_{t=0}=u_0,\ \ t\ge0,
\end{aligned}\right.
\end{align} 
where $$\ol W_k(t)=\tt W_{2k-1}(t),\ \ \hh W_k(t)=\tt W_{2k}(t),\ \ k\ge1.$$
Let $\U$ be a separable Hilbert space with a complete orthonormal basis $\{e_k\}_{k\ge 1}$. Let  
\begin{equation}\label{GH}
\left\{\begin{aligned}
& G(u)=-(u\cdot\nn)u+ \frac12\sum_{k=1}^{\infty}\left[\left(a_k \A_k\right)^2 u+\left(b_k \B_k\right)^2 u\right],\ \ k\ge 1,\\
& H(t,u)e_{3k-2} := a_k\A_k u,\ \ H(t,u)e_{3k-1} := b_k \B_k X,\ \ h(t,u)e_{3k}:= h_k(t,u),\ \ k\ge 1,\\
& \W(t):=\sum_{k=1}^\infty \left(\ol W_k(t) e_{3k-2} +\hh W_k(t) e_{3k-1}+\tt W_k(t) e_{3k}\right).
\end{aligned}\right.
\end{equation}
With the above notations, \eqref{Cauchy problem-Ito-AB}  reduces to
\begin{equation}\label{Cauchy problem G H}
\d u=\,\left[G(u)-F(u)\right] \d t+H(t,u)\d \W,\ \
u\big|_{t=0}=\,u_{0},\ \ t>0.
\end{equation}
Let $d\geq 2$ and recall $p_0$ in \eqref{p0}. Let
$s>\frac{d}{2}+1+\max\{2p_0,1\}$. According to Hypotheses \ref{H-Q} and \ref{H-hk}, if $u\in H^s$, then $G(u)\in H^{(s-1)\wedge(s-2p_0)}$ and $H(t,u)\in \LL_2(\U;H^{s-p_0})$, while by Lemma \ref{F-EP Lemma}, $F(u)\in H^s$. To apply the theory for SDEs in Hilbert space, we need to mollify $G(u)$ and $H(t,u)$. To this end, we will use the mollifier $J_n$ defined in \eqref{Define Jn} and construct the following regularization
\begin{equation}\label{GH-n}
\left\{\begin{aligned}
& G_n(u):= -J_n [(J_n u\cdot\nn J_n u)]+\frac{1}{2} \sum_{k=1}^{\infty}J_n^3(a_k\A_k)^2J_n u+\frac{1}{2} \sum_{k=1}^{\infty}J_n^3(b_k\B_k)^2J_n u,\  k\ge 1,\\
& H_n(t,u)e_{3k-2} := J_n(a_k\A_k) J_n u,\ \ H_n(t,u)e_{3k-1} := J_n(b_k \B_k) J_n u,\ \ H_n(t,u)e_{3k}:=  h_k(t,u),\  k\ge 1.
\end{aligned}\right.
\end{equation} 

We also need a cut-off function to split the expectation. Hence for any $R>1$, we take a cut-off function $\chi_R\in C^{\infty}([0,\infty);[0,1])$ such that 
\begin{equation}\label{chi-R}
\chi_R(y)=1\ \text{for}\ |y|\le R,\ \text{and}\ \chi_R(y)=0\ \text{for}\ y>2R,
\end{equation}
and then we consider
\begin{align}\label{Appro Cut-off}
\d u=\, \chi^2_R\big(\|u(t)-u_0\|_{\Wlip}\big)\left[G_n(u)-F(u)\right] \d t+\chi^2_R\big(\|u(t)-u_0\|_{\Wlip}\big)H_n(t,u)\d \W,\ \ 
u\big|_{t=0}=\,u_{0}.\end{align}

Keep in mind that  $p_0$ is in \eqref{p0} and we have the following 

%

\begin{Lemma}\label{Gn Hn Lemma} Let $d\geq 2$ and $s>\frac{d}{2}+1+\max\{2p_0,1\}$.  Let Hypotheses \ref{H-Q} and 
\ref{H-hk} be verified.  For any $R>1$, $n\ge1$ and $\F_0$-measurable $H^s$-valued random variable $u_0$, \eqref{Appro Cut-off} has a unique global solution $u_n=u_n^{(R)}(t,x)\in C([0,\infty);H^s)$ such that for some function 
$V:[0,\infty)\times[0,\infty)\rightarrow(0,\infty)$ increasing in both variables,
\begin{align}
\sup_{n\ge1}\E\Big[\sup_{t\in[0,T]}\|u_n\|^2_{H^s}\big|\F_0\Big]
\leq V(T,2R+\|u_0\|_{\Wlip})(1+\|u_0\|_{H^s}^2),\ \ T,R>0.\label{un uniform bound} 
\end{align} 
\end{Lemma} 

\begin{proof} 
By Lemma \ref{Lemma-Je}, it is easy to see that $G_n:H^s\to H^s$ and $H_n:[0,\infty)\times H^s\to \LL_2(\U;H^s)$ is locally Lipschitz.  Hence for any deterministic initial data \eqref{Appro Cut-off} admits a unique solution,  and the solution is continuous in $H^s$  (See for instance \cite{Prevot-Rockner-2007-book,Wang-2013-Book}).  
Combining this and the fact that
$\F_0$ is independent of the equation, we see that for any $\F_0$-measurable $H^s$-valued random variable $u_0$, 
\eqref{Appro Cut-off} also admits a unique solution $u_n=u_n(t)$, which is continuous in $H^s$. 

Now we verify \eqref{un uniform bound}. To begin with, we can infer from
\eqref{GH-n}, Hypothesis  \ref{H-hk}, Lemma \ref{Lemma-Je}, Theorem \ref{Cancel-Bk} and \ref{Cancel-Ak} (with $\mathcal{P}_n\equiv \D^s$, $f=J_n u$) that
\begin{align*}
&\sum_{k=1}^\infty \IP{H_n(t,u_n)e_{k}, u_n }_{H^s}^2\\
= \, &\sum_{k=1}^\infty
\left(
\IP{ J_n(a_k\A_k)J_n u_n,u_n }_{H^s}
+\IP{ J_n(b_k\B_k)J_n u_n,u_n }_{H^s}
+\IP{ h_k(t,u_n), u_n }_{H^s}^2\right) \\
\lesssim \,& \left(1+K^2(t,\|u_n\|_{\Wlip})\right)(1+\|u_n\|^4_{H^s}).
\end{align*}
Besides, it follows from \eqref{GH-n} that
\begin{align*}
&2\IP{ G_n(u_n),u_n }_{H^s}+\|H_n(t,u_n)\|^2_{\LL_2(\U;H^s)}\\
=\,& -2\IP{ J_n[(J_nu_n\cdot\nn)J_n u_n],u_n }_{H^s}
+\sum_{k=1}^{\infty}\IP{ J_n^3(a_k\A_k)^2J_n u_n,u_n }_{H^s}
+\sum_{k=1}^{\infty}\IP{ J_n^3(b_k\B_k)^2J_n u_n,u_n }_{H^s}\\
&+\sum_{k=1}^{\infty}\|J_n(a_k\A_k)J_n u_n\|^2_{H^s}
+\sum_{k=1}^{\infty}\|J_n(b_k\B_k)J_n u_n\|^2_{H^s}
+\sum_{k=1}^{\infty}\|h_k(t,u_n)\|^2_{H^s}\\
:=\,& \sum_{i=1}^{6} I_i.
\end{align*}
On account of  Hypothesis \ref{H-hk}, Lemmas \ref{Lemma-Je} and \ref{Kato-Ponce commutator estimate}, it holds that
\begin{align*}
|I_1|
\leq 2\left|\IP{ [(J_nu_n\cdot\nn)J_n u_n], J_nu_n }_{H^s}\right|
\lesssim \|u_n\|_{\Wlip}\|u_n\|^2_{H^s},\ \ |I_6|\leq K(t,\|u_n\|_{\Wlip})(1+\|u_n\|^2_{H^s}).
\end{align*}
It follows from Theorem \ref{Cancel-Bk} and \ref{Cancel-Ak} (with $\mathcal{P}_n=\D^sJ_n$ and $f=J_n u_n$) that
\begin{align*}
&\left|I_2+I_4\right|
+\left|I_3+I_5\right|
\lesssim\,
\|J_n u\|_{H^s}^2\leq \|u_n\|_{H^s}^2. 
\end{align*}
By the above estimates and It\^o's formula, we obtain that for some  function $\tt V:[0,\infty)\times[0,\infty)\rightarrow(0,\infty)$,
\begin{align*} 
{\rm d}\|u_n(t)\|^2_{H^s}-\d M_n(t) 
= \chi^2_R \big(\|u_n(t)-u_0\|_{\Wlip} \big)\bigg\{ \sum_{i=1}^{6} I_i\bigg\} \d t 
\le  \tt V(t,2R+\|u_0\|_{\Wlip})(1+\|u_n(t)\|_{H^s}^2)\d t,\end{align*}
where 
$$ \d M_n(t):=\, 2\chi^2_R \big(\|u_n(t) - u_0\|_{\Wlip}\big)\bIP{u_n(t), \, H_n(t,u_n)\d \W(t)}_{H^s}$$ satisfies 
$$\d \bIP{M_n (t)} \le \tt V(t,2R+\|u_0\|_{\Wlip})(1+\|u_n(t)\|_{H^s}^4\big)\d t.$$
Define
$$\tau_n:=\lim_{N\to\infty}\tau_{n,N},\ \ 
\tau_{n,N}:=\inf\big\{t\ge 0: \|u_n(t)\|_{H^s}\ge N\big\},\ \ n,\, N\ge 1.$$
For any $T>0$,  we  use BDG's inequality to find constants $c_1,c_2>0$  such that for any $t\in [0,T]$ and $N\ge 1$, 
\begin{align*}
&\E\bigg[\sup_{t'\in[0,t\land\tau_{n,N}]}\|u_n(t)\|^2_{H^s}\Big|\F_0\bigg]- \|u_0\|^2_{H^s}\\
\le \,&
c_1 \E\bigg[\bigg(\int_0^{t\land\tau_{n,N}}\tt V(t',2R+\|u_0\|_{\Wlip})
\Big(1+\|u_n(t')\|_{H^s}^4\Big)\d t'\bigg)^{\frac 1 2}\bigg|\F_0\bigg] \\
&+ c_1 \E\bigg[\int_0^{t\land\tau_{n,N}}\tt V(t',2R+\|u_0\|_{\Wlip}) \Big(1+\|u_n(t')\|_{H^s}^2\Big)\d t'\bigg|\F_0\bigg]\\
\le\, & \frac 1 2 \E\Big[\sup_{t'\in [0,t\land\tau_{n,N}]} \|u_n(t')\|_{H^s}^2\Big|\F_0\Big] + c_2 + c_2 \int_0^t\tt V(t',2R+\|u_0\|_{\Wlip}) \E\bigg[\sup_{r\in [0,t'\land \tau_{n,N}]} \|u_n(r)\|_{H^s}^2\Big|\F_0\bigg]\d t'.
\end{align*}
By Gr\"{o}nwall's inequality, there exists a function $V: [0,\infty)\times [0,\infty)\to (0,\infty)$ increasing in both variables such that
\begin{equation}\label{EXN} \E\Big[\sup_{t\in[0,T\land\tau_{n,N}]}\|u_n(t)\|^2_{H^s}\Big|\F_0\Big]\le V(T,2R+\|u_0\|_{\Wlip})(1+\|u_0\|_{H^s}^2),\ \ n,N\ge 1.
\end{equation} 
This implies that for all $n,N\ge1$,
\begin{align*}
\p\left(\tau_{n,N}<T\big|\F_0\right)\le 
\p\bigg(\sup_{t\in[0,T\land\tau_{n,N}]}\|u_n(t)\|_{H^s} 
\ge N\bigg|\F_0\bigg) 
\le \frac{V(T,2R+\|u_0\|_{\Wlip})(1+\|u_0\|_{H^s}^2)}{N^2},
\end{align*} 
so that $\tau_n=\lim_{N\to\infty}\tau_{n,N}$ satisfies
$$\p\big(\tau_n<T\big|\F_0\big)\le \lim_{N\to\infty} \p(\tau_{n,N}<T|\F_0) =0.$$
Hence, $\p\big(\tau_n\ge T\big)= \E\big[\p(\tau_n\ge T|\F_0)\big]=1$ for all $T>0$, which means  $\p(\tau_n=\infty)=1$. Letting $N\to\infty$ in \eqref{EXN} yields 
\eqref{un uniform bound}. 
\end{proof}

\subsection{Solving the cut-off problem}
In this section we will take limit in \eqref{Appro Cut-off} to find a solution to the following cut-off problem
\begin{equation} \label{Cut-off problem}
\d u=\chi^2_R\big(\|u(t)-u_0\|_{\Wlip}\big)\left[G(u)-F(u)\right] \d t+\chi^2_R\big(\|u(t)-u_0\|_{\Wlip}\big)H(t,u)\d \W,\ \
u\big|_{t=0}=\,u_{0},\ \ t>0,
\end{equation}
where  $\chi_R$, $F$ and $(G,H)$ are given in \eqref{chi-R}, \eqref{F-EP} and \eqref{GH}, respectively.

\begin{Lemma}\label{Lemma:Convergence of X-n}
Let $u_n$ be the approximate solution as in  Lemma \ref{Gn Hn Lemma}.
For any $n,l\ge 1$, $\delta_0\in\big(\frac{d}{2}+1,s-\max\{2p_0,1\}\big)$ and $T,\, N>0$, let
$$\tau^{n,l,T}_N:= T\land \inf\big\{t\ge 0: \|u_n(t)\|_{H^s}\lor\|u_l(t)\|_{H^s}\ge N\big\}.$$
Then $\pas$,
\begin{equation} \label{Cauchy in E-M}
\lim_{n\rightarrow\infty}\sup_{l\ge n}
\E\bigg[\sup_{t\in[0,\tau^{n,l,T}_N]} \|u_n(t)-u_l(t)\|^2_{H^{\delta_0}}\bigg|\F_0\bigg]=0,\ \ \ T,\, N>0.
\end{equation}
\end{Lemma}

\begin{proof}
Let $v_{n,l}=u_n-u_l$ for $n,l\ge 1$. 
We have that
\begin{equation}\label{Z-n-m equation}
\d v_{n,l}(t)=\sum_{i=1}^4 A_i^{n,l}(t)\d t+\sum_{i=1}^2 B_i^{n,l}(t) {\rm d}\W(t),\ \ v_{n,l}(0)=0, 
\end{equation}
where
\begin{align*}
&A_1^{n,l}(t) := -\left[\chi^2_R\big(\|u_n(t)-u_0\|_{\Wlip} \big)
-\chi^2_R\big(\|u_l(t)-u_0\|_{\Wlip}\big)\right]F(t,u_n(t)),\\
&A_2^{n,l}(t):=-\chi^2_R\big(\|u_l(t)-u_0\|_{\Wlip}\big)\left[F(t,u_n(t))- F(t,u_l(t))\right],\\
&A_3^{n,l}(t):=\left[\chi^2_R\left(\|u_n(t)-u_0\|_{\Wlip}\right)
-\chi^2_R\big(\|u_l(t)-u_0\|_{\Wlip}\big)\right]G_n(t,u_n(t)),\\
&A_4^{n,l}(t):= \chi^2_R\big(\|u_l(t)-u_0\|_{\Wlip}\big) \left[G_n(t,u_n(t))- G_l(t,u_l(t))\right],\end{align*} and 
\begin{align*}
&B_{1}^{n,l}(t):= \left[\chi_R(\|u_n(t)-u_0\|_{\Wlip})-\chi_R\big(\|u_l(t)-u_0\|_{\Wlip}\big)\right]H_n(t,u_n(t)),\\
&B_2^{n,l}(t):= \chi_R\big(\|u_l(t)-u_0\|_{\Wlip}\big)
[H_n(t,u_n(t))-H_l(t,u_l(t))]. 
\end{align*}
By the It\^{o} formula, we obtain 
\begin{align*}
\d\left\|v_{n,l}(t)\right\|^2_{H^{\delta_0}} = \, & 2\sum_{i=1}^2\bIP{v_{n,l}(t),\ B_i^{n,l}(t) \d\W(t) }_{H^{\delta_0}}\\
&+ \bigg\{\sum_{i=1}^2\left\|B_i^{n,l}(t)\right\|_{\mathcal L_2(\U;H^{\delta_0})}^2 + 2 \sum_{i=1}^4 \bIP{A_i^{n,l}(t), v_{n,l}(t)}_{H^{\delta_0}}\bigg\}\d t. 
\end{align*}
\textbf{Claim:} There is  a function $Q: [0,\infty)\times [0,\infty)\to (0,\infty)$ increasing in both variables and a function $\lambda: \mathbb N\times\mathbb N \rightarrow[0,\infty)$ with $\displaystyle \lim_{n,l\to\infty} 
\lambda_{n,l}=0$ such that  for all $n,l\ge 1$,
\begin{align}
&\sum_{i=1}^2\sum_{k=1}^\infty \bIP{v_{n,l}(t), B_i^{n,l}(t) e_k}_{H^{\delta_0}}^2 \le Q(t,N) \left\|v_{n,l}(t)\right\|_{H^{\delta_0}}^2
\left\{\lambda_{n,l}+\left\|v_{n,l}(t)\right\|_{H^{\delta_0}}^2\right\},\ \ t\in [0,\tau^{n,l,T}_N],\label{Convergence-1}\\
\sum_{i=1}^2&\left\|B_i^{n,l}(t)\right\|_{\mathcal L_2(\U; H^{\delta_0})}^2 + 2 \sum_{i=1}^4 \bIP{A_i^{n,l}(t),\, v_{n,l}(t)}_{H^{\delta_0}}\le Q(t,N)\left\{\lambda_{n,l}+\left\|v_{n,l}(t)\right\|_{H^{\delta_0}}^2\right\}, \ \ t\in [0,\tau^{n,l,T}_N].
\label{Convergence-2}
\end{align}
If \eqref{Convergence-1} and \eqref{Convergence-2} hold true, then
we use BDG's inequality to \eqref{Z-n-m equation} to find constants $a_1,a_2>0$ depending on $N$ and $T$ such that for all $n,l\ge 1$ and $t\in [0,N]$, 
\begin{align} &\E\bigg[\sup_{t'\in [0,t\land \tau^{n,l,T}_N]} \left\|v_{n,l}(t')\right\|_{H^{\delta_0}}^2\bigg|\F_0\bigg]\notag\\
\le\, & a_1 \E \bigg[\int_0^{t\land \tau^{n,l,T}_N}Q(t',N)\Big\{\lambda_{n,l}+\left\|v_{n,l}(t')\right\|_{H^{\delta_0}}^2\Big\}\d t'\bigg|\F_0\bigg]\notag\\
&+ a_1 \E\bigg[\bigg(\int_0^{t\land \tau^{n,l,T}_N}Q(t',N)\left\|v_{n,l}(t')\right\|_{H^{\delta_0}}^2\Big\{\lambda_{n,l}
+\left\|v_{n,l}(t')\right\|_{H^{\delta_0}}^2\Big\}\d t'\bigg)^{\frac 1 2}\bigg|\F_0\bigg]\notag\\
\le\, & \frac 1 2 \E\bigg[\sup_{t'\in [0,t\land \tau^{n,l,T}_N]} \left\|v_{n,l}(t')\right\|_{H^{\delta_0}}^2\bigg|\F_0\bigg]+ a_2 \lambda_{n,l}\notag\\ 
&+ a_2  \int_0^{t} Q(t',N)\E\bigg[\sup_{r\in [0,t'\land\tau^{n,l,T}_N]} \|v_{n,l}^{(R)}(r)\|_{H^{\delta_0}}^2\bigg|\F_0\bigg]\d t'.\label{Convergence-Ito}
\end{align}
By Gr\"onwall's inequality and noting $\lambda_{n,l}\to 0$ as $n,l\to\infty$, we prove \eqref{Cauchy in E-M}. Therefore now it suffices to prove \eqref{Convergence-1} and \eqref{Convergence-2}.  

We only prove \eqref{Convergence-2} since \eqref{Convergence-1} can be verified similarly. We note that $\chi_R(\cdot)$ is bounded and Lipschitz, $F(\cdot)$ is locally Lipschitz (cf. Lemma \ref{F-EP Lemma}) and $H^{\delta_0}\hookrightarrow \Wlip$. Then we use Hypothesis \ref{H-Q}, \eqref{GH-n} and Lemma \ref{Lemma-Je} to obtain that for all $n,l\ge 1$,
\begin{equation*}
\left\|B_1^{n,l}(t)\right\|_{\mathcal L_2(\U; H^{\delta_0})}^2 + 2 \sum_{i=1}^3 \IP{A_i^{n,l}(t),\, v_{n,l}(t)}_{H^{\delta_0}}\le Q(t,N)\left\|v_{n,l}(t)\right\|_{H^{\delta_0}}^2,\ \ t\in [0,\tau^{n,l,T}_N]
\end{equation*}
for some increasing function $Q: [0,\infty)\times [0,\infty)\to (0,\infty)$ increasing in both variables. 
Once again, since $\chi_R(\cdot)\le 1$,   we only need to prove that for all $n,l\ge 1$ and $t\in [0,\tau^{n,l,T}_N]$, 
\begin{align}
2\bIP{ G_n(u_n)-G_l(u_l), v_{n,l} }_{H^{\delta_0}}+ 
\big\| H_n(t,u_n)-H_l(t,u_l)\big\|^2_{\LL_2(\U;H^{\delta_0})}  
\le\,  
Q(t,N)
\left\{\lambda_{n,l}+\left\|v_{n,l}(t)\right\|_{H^{\delta_0}}^2\right\}.\label{Convergence-2-Q}
\end{align}
To this end, we find
$$2\bIP{ G_n(u_n)-G_l(u_l), v_{n,l} }_{H^{\delta_0}}+ 
\big\| H_n(t,u_n)-H_l(t,u_l)\big\|^2_{\LL_2(\U;H^{\delta_0})} =\Psi_{1} +\sum_{k=1}^\infty \sum_{i=2}^6\Psi_{i,k},$$ where 
\begin{align*} 
\Psi_{1} =\Psi^{n,l}_{1} :=\, &2\IP{J_n[(J_n u_n\cdot\nn)J_n u_n]-J_l[ (J_l u_l\cdot\nn)J_l u_l],\, u_n-u_l}_{H^{\delta_0}}\\
\Psi_{2,k}=\Psi^{n,l}_{2,k}:=\, & \IP{J_n^3(a_k\A_k)^2J_n u_n-J_l^3(a_k\A_k)^2J_l u_l,\, u_n-u_l }_{H^{\delta_0}},\\ 
\Psi_{3,k}=\Psi^{n,l}_{3,k}:=\, & \IP{J_n^3(b_k\B_k)^2J_n u_n-J_l^3(b_k\B_k)^2J_l u_l,\,  u_n-u_l}_{H^{\delta_0}}, \\
\Psi_{4,k}=\Psi^{n,l}_{4,k}:=\, & \| h_n(t,u_n)e_{3k-2}-h_l(t,u_l)e_{3k-2}\|^2_{H^{\delta_0}},\\  
\Psi_{5,k}=\Psi^{n,l}_{5,k}:=\, & \| h_n(t,u_n)e_{3k-1}-h_l(t,u_l)e_{3k-1}\|^2_{H^{\delta_0}}, \\
\Psi_{6,k}=\Psi^{n,l}_{6,k}:=\, &\| h_k(t,u_n)-h_k(t,u_l)\|^2_{H^{\delta_0}}.  
\end{align*} 
For $\Psi_1$, one can show that for $\epsilon\in(0,s-1-\delta_0)$,
\begin{align*}
|\Psi_{1} |\lesssim
\left((\|u_n\|_{H^{s}}+\|u_l\|_{H^{s}})^4+1\right)\left(\|v_{n,l}\|^2_{H^{\delta_0}}
+(l\land n)^{-2(s-1-\delta_0-\epsilon)}\right).
\end{align*}
The proof for this estimate is similar to \cite[Lemma 3.1]{Tang-Yang-2022-AIHP}, 
and here we omit the details to save space.
It suffices to estimate the other two terms. To control  $\sum_{k=1}^\infty \left\{\Psi_{3,k}+\Psi_{5,k}\right\}$, we  find
$$
\Psi_{3,k}=\sum_{j=1}^{3}\Psi_{3,k,j},\ \ \ \Psi_{5,k}=\sum_{i,j=1}^{3}\bIP{\Psi_{5,k,i},\Psi_{5,k,j}}_{H^{\delta_0}},$$
where 
\begin{equation*}
\begin{cases}
\Psi_{3,k,1}:= \bIP{(J^3_n-J^3_l)(b_k\B_k)^2 J_n u_l, v_{n,l}}_{H^{\delta_0}},\ \ 
&\Psi_{5,k,1}:=\, (J_n-J_l)(b_k\B_k)J_n u_n,\\
\Psi_{3,k,2}:= \bIP{J^3_l(b_k\B_k)^2 (J_n-J_l) u_l, v_{n,l}}_{H^{\delta_0}},\ \
&\Psi_{5,k,2}:=\, J_l(b_k\B_k) (J_n-J_l) u_n,\\
\Psi_{3,k,3}:= \bIP{J^3_l(b_k\B_k)^2 J_l v_{n,l}, v_{n,l}}_{H^{\delta_0}},\ \
&\Psi_{5,k,3}:=\, J_l(b_k\B_k) J_l v_{n,l}.
\end{cases}
\end{equation*}
By Hypothesis \ref{H-Q} and Lemma \ref{Lemma-Je}, we have for any $\epsilon\in(0,s-2p_0-\delta_0)$,
\begin{equation*}         
\sum_{k=1}^\infty\Psi_{3,k,1},\ 
\sum_{k=1}^{\infty}\Psi_{3,k,2},\ 
\sum_{k=1}^{\infty}\sum_{i\in \{1,2\}}
\bIP{\Psi_{5,k,i},\Psi_{5,k,3}}_{H^{\delta_0}}
\lesssim\, (l\land n)^{-(s-2p_0-\delta_0-\epsilon)} 
\norm{u_n }_{H^s}\|v_{n,l}\|_{H^{\delta_0}},              
\end{equation*}
\begin{equation*}
\sum_{k=1}^{\infty}\sum_{i,j\in \{1,2\}}
\bIP{ \Psi_{5,k,i},\Psi_{5,k,j} }_{H^{\delta_0}} 
\lesssim\, (l\land n)^{-2(s_0-2p_0-\delta_0-\epsilon)} 
\norm{u_n}^2_{H^s},                            
\end{equation*}
Then we apply  Proposition \ref{Cancel-Bk} (with $s=\delta_0$, $\P_n=\D^{\delta_0}J_l$ and $f=J_l v_{n,l}$) to find 
$$\sum_{k=1}^{\infty}\left\{\Psi_{3,k,3}+\IP{\Psi_{5,k,3},\Psi_{5,k,3}}_{H^{\delta_0}}\right\}\lesssim \|v_{n,l}\|^2_{H^{\delta_0}}.$$
Hence  we find an increasing function $Q: [0,\infty)\times [0,\infty)\to (0,\infty)$ increasing in both variables such that
$$\sum_{k=1}^\infty \left\{\Psi_{3,k}+\Psi_{5,k}\right\}\lesssim Q(t,N)\big\{(l\land n)^{-(s_0-2p_0-\delta_0-\epsilon)}+\|v_{n,l}\|_{H^{\delta_0}}^2\big\},\ \ n,\,l\ge 1,\  t\in [0,\tau^{n,l,T}_N].$$
Similarly,   the same estimate holds for $\sum_{k=1}^{\infty}\left\{\Psi_{4,k}+\Psi_{6,k}\right\}$.  Obviously, the desired upper bound of $\Psi_{6,k}$ follows from Hypothesis \ref{H-hk}.
In conclusion, \eqref{Convergence-2-Q} holds true. 
\end{proof}

\begin{Lemma}\label{un-u pas}
Let$u_n$ be the approximate solution as in  Lemma \ref{Gn Hn Lemma}.
There exists an $\mathcal{F}_t$-progressive measurable $H^s$-valued process $u(t)=(u^{(R)}(t))_{t\ge 0}$ such that, up to a subsequence, $\p$-{\rm a.s.},
\begin{equation}\label{Xn to X}
u_n\xrightarrow[]{n\to \infty}u \ {\rm in}\ C([0,\infty);H^{\delta_0}).
\end{equation}
\end{Lemma}

\begin{proof} 
For any $T>0$, $N\ge 1$ and $\epsilon>0$, by using \eqref{un uniform bound} in Lemma \ref{Gn Hn Lemma} and Chebyshev's inequality, 
we have
\begin{align*} 
&\p(\tau^{n,l,T}_N<T|\F_0)\\
\le\, & \p\bigg(\sup_{t\in [0,T]} \|u_n(t)\|_{H^s}\ge N\bigg|\F_0\bigg)+ \p\bigg(\sup_{t\in [0,T]} \|u_l(t)\|_{H^s}\ge N\bigg|\F_0\bigg)\\
\le\, & \frac {2 V(T,2R+\|u_0\|_{\Wlip})(1+\norm{u_0}^2_{H^s})}{N^2}.
\end{align*} 
Since $\tau^{n,l,T}_N\le T$ $\p$-a.s.,
for any $T>0,\, N>\ge1$, $n,\, l\ge 1$, we have
\begin{align*}
&\p\bigg(\sup_{t\in[0,T]}\|u_n(t)-u_l(t)\|_{H^{\delta_0}}>\epsilon\bigg|\F_0\bigg)\\
\le\, & \p\left(\tau^{n,l,T}_N<T\big|\F_0\right) + \p\bigg(\sup_{t\in[0,\tau^{n,l,T}_N]}\|u_n(t)-u_l(t)\|_{H^{\delta_0}}>\epsilon\bigg|\F_0\bigg)\\
\leq\, &\frac {2 V(T,2R+\|u_0\|_{\Wlip})(1+\norm{u_0}^2_{H^s})}{N^2}
+\p\bigg(\sup_{t\in[0,\tau^{n,l,T}_N]}\|u_n(t)-u_l(t)\|_{H^{\delta_0}}>\epsilon\bigg|\F_0\bigg).
\end{align*} On account of Lemma \ref{Lemma:Convergence of X-n}, we first let $n,l\to\infty$ and then $N\to\infty$ to find
$$ \lim_{n,l\rightarrow\infty} \p\bigg(\sup_{t\in[0,T]}\|u_n(t)-u_l(t)\|_{H^{\delta_0}}>\epsilon\bigg|\F_0\bigg)=0,\ \ \epsilon,\,T>0.$$ 
According to the reverse Fatou lemma, this gives rise to
\begin{align*} &\limsup_{n,l\rightarrow\infty} \p\bigg(\sup_{t\in[0,T]}\|u_n(t)-u_l(t)\|_{H^{\delta_0}}>\epsilon\bigg)\\
= \, &\limsup_{n,l\rightarrow\infty} \E\bigg[\p\bigg(\sup_{t\in[0,T]}\|u_n(t)-u_l(t)\|_{H^{\delta_0}}>\epsilon\bigg|\F_0\bigg)\bigg]\\
\le \, & \E\bigg[\limsup_{n,l\rightarrow\infty} \p\bigg(\sup_{t\in[0,T]}\|u_n(t)-u_l(t)\|_{H^{\delta_0}}>\epsilon\bigg|\F_0\bigg)\bigg]=0,\ \ \epsilon,\, T>0.
\end{align*} 
Therefore, up to a subsequence, \eqref{Xn to X} holds for certain progressively measurable process $u$ on $H^s$.
\end{proof}

\begin{Lemma}\label{cut-off solution} Let $d\geq 2$ and $s>\frac{d}{2}+1+\max\{2p_0,1\}$.  Let Hypotheses \ref{H-Q} and 
\ref{H-hk} hold.  For any $R>1$, $n\ge1$ and $\F_0$-measurable $H^s$-valued random variable, \eqref{Cut-off problem} has a unique global solution $u=u^{(R)}$ such that for any $T>0$, 
\begin{equation}\label{limit continuous}
\p\big(u\in C([0,T];H^s)\big)=1,
\end{equation}
and
\begin{equation}\label{X L2 bound}
\E\Big[\sup_{t\in [0,T]} \|u(t)\|_{H^s}^2\big|\F_0\Big]\le V(T,2R+\|u_0\|_{\Wlip})(1+\norm{u_0}^2_{H^s}),
\end{equation} 
where $V$ is given in Lemma {\rm \ref{Gn Hn Lemma}}.
\end{Lemma}

\begin{proof}
For any $R\ge 1$,
by Lemma \ref{un-u pas}, we can take limit to see that the limit process $u$ obtained in
Lemma \ref{un-u pas} is a solution to \eqref{Cut-off problem}. Uniqueness of $u$ can be obtained in the same way as we estimate \eqref{Convergence-Ito}.
Besides, \eqref{X L2 bound} comes from
\eqref{Xn to X} and
\eqref{un uniform bound}.

Now we prove \eqref{limit continuous}.
By \eqref{Xn to X}, we know that $u\in C([0,T];H^{\delta_0})$, which, together with the fact that $H^s\hookrightarrow H^{\delta_0}$ is dense, means that $u$ is weakly continuous in $H^{s}$. In order to prove \eqref{limit continuous}, 
we only need to prove that $[0,T]\ni t\mapsto\|u(t)\|_{H^s}$ is continuous almost surely. Let
\begin{equation}\label{TNN}
\tau_N:=N\land \inf\big\{t\ge 0: \|u(t)\|_{H^s}\ge N\big\},\ \ N\ge 1.\end{equation}
Note that \eqref{X L2 bound} implies $\lim_{N\to\infty} \tau_N=\infty$ $\pas$
It suffices to prove 
\begin{equation}\label{CTT} 
\|u(\cdot)\|_{H^s}\in C([0,\tau_N\wedge T];\R),\ \ \ N\ge 1.
\end{equation} 
However, since we only know $u\in H^s$ (by \eqref{X L2 bound}), one cannot use  It\^{o}'s formula to $\|u\|^2_{H^s}$ since $\IP{H(t,u)e_{k}, u }_{H^s}$ and $\bIP{\left[G(u)-F(u)\right], u }_{H^s}+\|H(t,u)\|^2_{\LL_2(\U;H^s)}$ are not well-defined.
Then we apply Lemma \ref{Te uux+F(u)}, 
\eqref{GH}, Hypothesis  \ref{H-hk}, Lemma \ref{Lemma-Je}, Theorem \ref{Cancel-Bk} and \ref{Cancel-Ak} (with $\mathcal{P}_n= \D^sJ_n$, $f=u$) to find
\begin{align*}
&\sum_{k=1}^\infty \IP{J_n H(t,u)e_{k}, J_n u }_{H^s}^2\\
= \, &\sum_{k=1}^\infty
\left(
\IP{ J_n(a_k\A_k) u, J_nu }_{H^s}
+\IP{ J_n(b_k\B_k) u,J_n u }_{H^s}
+\IP{ J_n h_k(t,u), J_n u }_{H^s}^2\right) \\
\lesssim \,& \left(1+K^2(t,\|u\|_{\Wlip})\right)(1+\|u\|^4_{H^s}),
\end{align*}
and 
\begin{align*}
&\left|2\bIP{J_n \left[G(u)-F(u)\right],J_n u }_{H^s}+\|J_n H(t,u)\|^2_{\LL_2(\U;H^s)}\right|\\
\leq\,& 2\varLambda \|u\|^2_{H^s}\|u\|_{W^{1,\infty}}
+\left|\sum_{k=1}^{\infty}\IP{ J_n(a_k\A_k)^2 u, J_nu }_{H^s}
+\sum_{k=1}^{\infty}\|J_n(a_k\A_k)u\|^2_{H^s}\right|\\
&+\left|\sum_{k=1}^{\infty}\IP{ J_n(b_k\B_k)^2 u, J_n u }_{H^s}
+\sum_{k=1}^{\infty}\|J_n(b_k\B_k) u\|^2_{H^s}\right|
+\sum_{k=1}^{\infty}\|h_k(t,u)\|^2_{H^s}\\
\lesssim\,&  \left(1+\|u\|_{\Wlip}+K^2(t,\|u\|_{\Wlip})\right)(1+\|u\|^4_{H^s}).
\end{align*}
Therefore,  by It\^{o}'s formula to $\|J_n u\|^2_{H^s}$ for \eqref{Cut-off problem}, for any $n,\,N\ge 1$ we find a martingale $M_t^{(n)}$ such that for some constant  $Q_N>0$ such that 
\begin{equation}\label{C0T} -Q_N \d t\le \d \|J_n u(t)\|_{H^s}^2 + \d M^{(n)}(t) \le Q_N \d t,\ \
\d \bIP{M^{(n)}}(t)\le Q_N \d t,\ \ t\in [0,\tau_N],\ \ n\ge 1.\end{equation}
That is to say, for some constant $C_N>0$, 
$$\E\Big[\big|\|J_n u(t\land\tau_N)\|_{H^s}^2-\|J_n u(t'\land\tau_N)\|_{H^s}^2\big|^4\Big]\le C_N|t-t'|^2,\ \ \ t,t'\ge 0,\  \ n\ge 1.$$
By Lemma \ref{Lemma-Je} and Fatou's lemma with $n\to\infty$, we derive 
$$\E\Big[\big|\| u(t\land\tau_N)\|_{H^s}^2-\|  u(t'\land\tau_N)\|_{H^s}^2\big|^4\Big]\le C_N|t-t'|^2,\ \ t,s\ge 0.$$
From this and Kolmogorov's continuity theorem, we obtain \eqref{CTT}.  
\end{proof} 

\subsection{Finish the proof for Theorem \ref{Thm-EP}}

Now we are in the position to prove Theorem \ref{Thm-EP}.  

\begin{proof}[Proof for Theorem \ref{Thm-EP}] \ref{T1}.  
Let $u=u^{(R)}$ be the solution to \eqref{Cut-off problem} as in Lemma \ref{cut-off solution}. Now we remove the cut-off. To this end,  we let 
\begin{equation}
\tau^{(R)}:= \inf\big\{t\ge 0: \|u^{(R)}(t)-u_0\|_{\Wlip}\ge R\big\}.\label{tau-R-M}
\end{equation}
By the continuity of $u^{(R)}(t)$ in $H^{\delta_0}$ and $H^{\delta_0}\hookrightarrow \Wlip$, we have $\p(\tau^{(R)}>0)=1$ for any $R>0$. Since $\chi_R^2\big(\|u^{(R)}(t)-u_0\|_{\Wlip}\big)=1$ for $t\le \tau$, 
$(u^{(R)},\tau^{(R)})$ is a local solution to \eqref{Cauchy problem-Ito} (or equivalently, \eqref{Cauchy problem-Ito-AB}). The uniqueness of $u^{(R)}$ 
implies
$$u^{(R)}(t)= u^{(R+1)}(t),\ \ t\le \tau^{(R)}, \ R\ge 1\ \ \p\text{-a.s.}$$
Let $\tau^* :=\lim_{R\to\infty} \tau^{(R)}$, $\tau^{(0)}:=0$ and we define
\begin{align} 
u(t):=\sum_{R=1}^\infty \textbf{1}_{[\tau^{(R-1)}, \tau^{(R)})}(t)u^{(R)}(t),\ \ t\in [0,\tau^*).\label{tau-R-M-X} 
\end{align}
Then one can conclude that $(u,\tau^*)$ is a local solution to \eqref{Cauchy problem-Ito}. 
Again, by the uniqueness of $u^{(R)}$ and \eqref{limit continuous}, $\p\big(u\in C([0,\tau^*);H^s)\big)=1$.
Moreover,  the construction of $\tau^*$ and \eqref{X L2 bound} immediately tell us $$\limsup_{t\to\tau^*}\|u(t)\|_{\Wlip}=\limsup_{t\to\tau^*}\|u(t)\|_{H^s}=\infty\ \text{on}\ \{\tau^*<\infty\}\ \ \pas,$$ which gives \eqref{blow-up criterion}. 

\medskip

\ref{T2}. Recalling \eqref{GH}, and then
using It\^o's formula to 
$\log({\rm e}+\|u(t)\|_{H^{\eta}}^2)$ with noting \eqref{uux F 1} in Lemma \ref{Te uux+F(u)} and \eqref{GH}, we arrive at
\begin{align}
&\d \log({\rm e}+\|u(t)\|_{H^{\eta}}^2)\notag\\
=\,&\frac{1}{{\rm e}+\|u(t)\|_{H^{\eta}}^2}\Big\{2\bIP{G(u(t))+F(u(t)), u(t)}_{H^{\eta}}+ \|H(t,u(t))\|_{\mathcal L_2(\U;H^\eta)}^2\Big\}\d t\notag\\
&-\frac{2}{\left({\rm e}+\|u(t)\|_{H^{\eta}}^2\right)^2}
\sum_{k=1}^\infty\bIP{H(t,u(t))e_k,u(t)}^2_{H^{\eta}}\d t+\d M_t\notag\\
\leq\, &\frac{1}{{\rm e}+\|u(t)\|_{H^{\eta}}^2}\bigg\{ 2\varLambda \|u(t)\|^2_{H^{\eta}}\|u(t)\|_{W^{1,\infty}}+ 
\sum_{k=1}^\infty\bigg|\IP{(a_k\A_k)^2 u(t), u(t)}_{H^{\eta}} + \left\|a_k\A_k u(t)\right\|_{H^{\eta}}^2\bigg| \notag\\
&\qquad\qquad\qquad\quad+ \sum_{k=1}^\infty\bigg|\IP{(b_k\B_k)^2 u(t), u(t)}_{H^{\eta}} + \left\|b_k\B_k u(t)\right\|_{H^{\eta}}^2\bigg| +
\sum_{k=1}^\infty\|h_k(t,u(t))\|_{H^\eta}^2\bigg\}\d t\notag\\
&-\frac{2}{\left({\rm e}+\|u(t)\|_{H^{\eta}}^2\right)^2}
\sum_{k=1}^\infty\bIP{h_k(t,u(t)),u(t)}^2_{H^{\eta}}\d t+\d M_t,\ \ 
t\in [0,\tau^*),\label{global V estimate}
\end{align}
where $M_t$ is a martingale up to $\tau_{N}$ defined in \eqref{TNN} for any $N\ge 1$. 
According to 
\eqref{NE SPDE} and  \eqref{NE h-k}, one can find a bounded function $Q:[0;\infty)\to(0,\infty)$ such that 
\begin{align*}
\frac{1}{{\rm e}+\|u(t)\|_{H^{\eta}}^2}\bigg\{ &2\varLambda \|u\|^2_{H^{\eta}}\|u\|_{W^{1,\infty}}+ C\|u(t)\|^2_{H^{\eta}} \\
&+
\sum_{k=1}^\infty\|h_k(t,u(t))\|_{H^\eta}^2-\frac{2}{\left({\rm e}+\|u(t)\|_{H^{\eta}}^2\right)}
\sum_{k=1}^\infty\bIP{h_k(t,u(t)),u(t)}^2_{H^{\eta}}\bigg\}\leq Q(t),\ \ t\in [0,\tau^*).
\end{align*}
Using this, Theorems \ref{Cancel-Bk} and \ref{Cancel-Ak} in \eqref{global V estimate} yields
\begin{align*}
&\d \log({\rm e}+\|u(t)\|_{H^{\eta}}^2)\\
\leq\, &\frac{1}{{\rm e}+\|u(t)\|_{H^{\eta}}^2}\bigg\{ 2\varLambda \|u\|^2_{H^{\eta}}\|u\|_{W^{1,\infty}}+ C\|u(t)\|^2_{H^{\eta}} +
\sum_{k=1}^\infty\|h_k(t,u(t))\|_{H^\eta}^2\bigg\}\d t\\
&-\frac{2}{\left({\rm e}+\|u(t)\|_{H^{\eta}}^2\right)^2}
\sum_{k=1}^\infty\bIP{h_k(t,u(t)),u(t)}^2_{H^{\eta}}\d t+\d M_t\\
\leq\,& Q(t)\d t+\d M_t,\ \ 
t\in [0,\tau^*),
\end{align*}
which means that for some function 
$V:[0,\infty)\times[0,\infty)\rightarrow(0,\infty)$ increasing in both variables,
\begin{align*}
\E\Big[\log({\rm e}+\|u(t\wedge\tau_N)\|_{H^{\eta}}^2)\big|\F_0\Big]
\leq V(t,\|u_0\|_{H^s}),\ \ 
t\ge0,\ \ N\ge1.
\end{align*}
Consequently, by the continuity of $u$ in $H^s$ (hence also in $H^{\eta}$), we derive 
\begin{align*} \p\big(\tau^*<t\big|\F_0\big) & \le \p\big(\tau_{N}<t\big|\F_0\big) \le \frac{ \E\Big[\log({\rm e}+\|u(t\wedge\tau_N)\|_{H^{\eta}}^2)\big|\F_0\Big]}{\log({\rm e}+N^2)} \le \frac{V(t,\|u_0\|_{H^s})}{\log({\rm e}+N^2)},\ \ N\ge 1,\ t>0.
\end{align*} 
Letting $N\to\infty$ and then $t\to\infty$ we see that $\p(\tau^*<\infty|\F_0)=0$ and hence, $\p(\tau^*<\infty)=0.$ 
\end{proof}

\section{Noise effect on the dependence on initial data}\label{Non uniform section}

In this section, we consider the problem \eqref{EP non stable Eq} on $\T^d$. For simplicity, we fix a separable Hilbert space $\U$ with the complete orthonormal basis $\{e_k\}_{k\ge1}$. Then we reformulate \eqref{EP non stable Eq} on $\T^d$ as
\begin{equation}\label{SEP non stable BdW}
\left\{\begin{aligned}
{\rm d}u +\left[\left(u\cdot\nabla\right) u+F(u)\right]\d t
=\,&B(t,u)\d\W(t),\ \ t>0,\ \ u\big|_{t=0}=u_0,\ \ x\in \T^d,\\
\W(t):=\, &\sum_{k=1}^\infty W_ke_k,\\
B(t,u)e_k:=\, &h_k(t,u).
\end{aligned}\right.
\end{equation}

We assume that  $h_k(t,\cdot)$ is controlled by $F$ in the following sense:
\begin{Hypothesis}\label{H-hk<F}
For all $k$, 
$h_k:[0,\infty)\times H^s\ni (t,u)\mapsto h_k(t,u)\in H^s$ is continuous for $s>\frac{d}{2}$, and
\begin{equation}\label{assumption G}
\sum_{k=1}^{\infty}\|h_k(t,u)\|^2_{H^s}\leq \|F(u)\|^2_{H^s},\ \ \sum_{k=1}^\infty \|h_k(t,u)- h_k(t, v)\|_{ H^s}^2 \leq \|F(u)-F(v)\|^2_{H^s},
\end{equation}
where $F$ is defined in \eqref{F-EP}.
\end{Hypothesis}

With the above notations at hand, Hypothesis \ref{H-hk<F} is equivalent to
\begin{ManualHypo}{$\bf H_5^\prime$}\label{H-B<F}
$B:(t,u) \mapsto B(t,u)\in \LL_2(\U; H^s)$ is continuous for $s>\frac{d}{2}$ and
\begin{equation}
\|B(t,u)\|_{\LL_2(\U; H^s)}\leq \|F(u)\|_{H^s},\ \ \|B(t,u)-B(t,v)\|_{\LL_2(\U;H^s)}\leq \|F(u)-F(v)\|_{H^s},\ \ s>\frac{d}{2}.
\end{equation}
\end{ManualHypo}

For \eqref{SEP non stable BdW}, we 
have the following 
\begin{Proposition}\label{SEP non stable BdW-T}
Let $s>\frac{d}{2}+1$. Let Hypothesis \ref{H-B<F} $($equivalently, Hypothesis \ref{H-hk<F}$)$ hold. If $u_0$ is an $H^s$-valued $\mathcal{F}_0$-measurable random variable with $\E\|u_0\|^2_{H^s}<\infty$, then there is a unique maximal solution $(u,\tau^*)$ to \eqref{SEP non stable BdW} in the sense of Definitions \ref{pathwise solution definition}, and $(u,\tau^*)$ satisfies \eqref{blow-up criterion}.
\end{Proposition}

\begin{proof}
Since Hypothesis \ref{H-hk<F} implies Hypothesis \ref{H-hk}, 
existence, uniqueness and the blow-up criterion \eqref{blow-up criterion} in $H^{s}$ with $s>\frac{d}{2}+2$ come from Theorem \ref{Thm-EP}. The extension from $s>\frac{d}{2}+3$ to $s>\frac{d}{2}+1$ can be done, as in \cite{Miao-Rohde-Tang-2021-arXiv,GlattHoltz-Vicol-2014-AP}, by mollifying initial data and then passing to the limit, as in the same way. Here we omit the
details to avoid redundancy.
\end{proof}

For the noise effect on the solution map $u_0\mapsto (u,\tau)$, we consider \eqref{EP non stable Eq} and we have

\begin{Theorem}\label{SEP non uniform}
Let $s>d/2+1$ with $d\geq2$. Let Hypothesis \ref{H-B<F} be satisfied. Then there is at least one of the following properties holding true for the problem \eqref{SEP non stable BdW}:
\begin{enumerate}[label={\bf (\roman*)}]

\item For any $R\gg1$, the $R$-exiting time is \textbf{not} strongly stable at the zero solution in the sense of Definition {\rm\ref{Stability on exiting time}}.

\item The solution map $u_0\mapsto u$ defined by \eqref{EP non stable Eq} is \textbf{not} uniformly continuous, as a map from $L^p(\Omega,H^s)$ $(p\in[1,\infty])$ into $L^1\left(\Omega; C\left([0,T];H^s\right)\right)$ for any $T>0$.
More precisely, there exist two sequences of solutions $u_{1,n}(t)$ and $u_{2,n}(t)$, and two sequences of stopping times $\tau_{1,n}$ and $\tau_{2,n}$, such that
\medskip
\begin{enumerate}[label={\bf (\alph*)}]

\item\label{NU depend 1} $\p\{\tau_{i,n}>0\}=1$ for each $n>1$ and $i=1,2$. Besides,
\begin{equation*}
\lim_{n\rightarrow\infty} \tau_{1,n}=\lim_{n\rightarrow\infty} \tau_{2,n}=\infty\ \ \pas
\end{equation*}

\item\label{NU depend 2} For $i=1,2$, $u_{i,n}\in C([0,\tau_{i,n}];H^s)$ $\pas$, and
\begin{equation*}
\Big\|\sup_{t\in[0,\tau_{1,n}]}\|u_{1,n}(t)\|_{H^{s}}\Big\|_{L^p(\Omega)}+
\Big\|\sup_{t\in[0,\tau_{2,n}]}\|u_{2,n}(t)\|_{H^{s}}\Big\|_{L^p(\Omega)}\lesssim 1,\ \ p\in[1,\infty].
\end{equation*}

\item\label{NU depend 3} At  $t=0$,
\begin{align*}
\lim_{n\rightarrow\infty}\|u_{1,n}(0)-u_{2,n}(0)\|_{L^p(\Omega;H^{s})}=0,\ \ p\in[1,\infty].
\end{align*}
\item\label{NU depend 4} For any $T>0$, we have
\begin{align*}
\liminf_{n\rightarrow\infty}\E
\sup_{t\in[0,T\wedge\tau_{1,n}\wedge\tau_{2,n}]}\|u_{1,n}(t)-u_{2,n}(t)\|_{H^{s}} \gtrsim \sup_{t\in[0,T]}|\sin 2\pi t|.
\end{align*}
\end{enumerate}

\end{enumerate}

\end{Theorem}

\begin{Remark} We give the following remarks concerning Theorem \ref{SEP non uniform}.

{\bf (1)} It is worthwhile noting that in deterministic cases, the issue of the optimal dependence of solutions (for example, the solution map is continuous but not uniformly continuous) to various nonlinear dispersive and integrable equations has been the subject of many papers. One of the first results of this type dates back at least as far as 
to Kato \cite{Kato-1975-ARMA}, where Kato proved that the solution map $H^s(\T)\ni u_{0}\mapsto u$ ($s>3/2$) given by the inviscid Burgers equation is not H\"{o}lder continuous regardless of the H\"{o}lder exponent.
Since then different techniques have been successfully applied to various problems. 
Particularly, for the incompressible Euler equation, we refer to \cite{Himonas-Misiolek-2010-CMP,Tang-Liu-2014-JMP}, and for CH type equations, we refer to \cite{Himonas-Kenig-2009-DIE,Himonas-Kenig-Misiolek-2010-CPDE,Tang-Liu-2015-ZAMP,Tang-Shi-Liu-2015-MM,Tang-Zhao-Liu-2014-AA} and the references therein. 

{\bf (2)} To prove Theorem \ref{SEP non uniform}, we assume that for some $R_0\gg1$, the $R_0$-exiting time of the zero solution is strongly stable. Then we will construct an example to show that the solution map $u_0\mapsto u$ defined by \eqref{EP non stable Eq} is not uniformly continuous. This example involves the construction (for each $s>d/2+1$) of two sequences of solutions which are converging at time zero but remain far apart at any later time. Actually, we will first construct two sequences of approximation solutions $u^{l,n}(l\in\{-1,1\})$ such that the actual solutions $u_{l,n}(l\in\{-1,1\})$ starting from $u_{l,n}(0)=u^{l,n}(0)$ satisfy that as $n\rightarrow\infty$,
\begin{align}
\lim_{n\rightarrow\infty}\E\sup_{[0,\tau_{l,n}]}\|u_{l,n}-u^{l,n}\|_{H^s}^2
=0,\label{non uniform remark equ}
\end{align}
where $u_{l,n}$ exists at least on $[0,\tau_{l,n}]$. Due to the lack of life span estimate in stochastic setting, in order to obtain \eqref{non uniform remark equ}, we first connect the property $\inf_{n}\tau_{l,n}>0$ with the stability property of the exiting time of the zero solution. In deterministic case, we have uniform lower bounds for the existence times of a sequence of solutions (see (4.7)--(4.8) in \cite{Tang-Shi-Liu-2015-MM} and (3.8)--(3.9) in \cite{Tang-Zhao-Liu-2014-AA} for example). If \eqref{non uniform remark equ} holds true, then we can estimate the approximation solutions instead of the actual solutions and obtain \ref{NU depend 4} by showing that the error in $H^{2s-\sigma}$ behaves like $n^{s-\sigma}$, but the error in $H^{\sigma}$ is $O(1/n^{r_s})$, where $d/2<\sigma<s-1$ and $-r_s+s-\sigma<0$. These two estimates and interpolation give \eqref{non uniform remark equ}. Theorem \ref{SEP non uniform} is proved for $d\geq 2$. However, the proof holds true also for $d=1$, namely the stochastic CH equation case (see Remark \ref{SEP non uniform remark}). 
%

{\bf (3)} Theorem \ref{SEP non uniform} implies that for the issue of the dependence on initial data, we cannot expect  the multiplicative noise (in It\^{o} sense) to improve the stability of the exiting time of the zero solution, and simultaneously improve the continuity of the dependence on initial data. Formally speaking, the ``regularization by (It\^{o} sense) noise'' actually preserves the hyperbolic structure of the equations.
As for the noise in the sense of Stratonovich, whether it can improve the dependence on initial data is our future work.

\end{Remark}

Now we proceed to prove Theorem \ref{SEP non uniform}. We assume that for some $R_0\gg1$, the $R_0$-exiting time is strongly stable at the zero solution. Then we will show that the solution map $u_0\mapsto u$ defined by \eqref{EP non stable Eq} is not uniformly continuous. We will firstly assume that the dimension $d\geq 2$ is even.

\subsection{Estimates on the errors}

Let $l\in\{-1,1\}$. Define divergence-free vector field as
\begin{equation}\label{approximate solutions definition 1}
u^{l, n}=(l n^{-1}+n^{-s}\cos\theta_{1},\, l n^{-1}+n^{-s}\cos\theta_{2},\,\cdots,\, l n^{-1}+n^{-s}\cos\theta_{d}),
\end{equation}
where $\theta_{i}=2\pi (nx_{d+1-i}-l t)$ with $1\leq i\leq d$ and $n\ge1$.
Substituting $u^{l,n}$ into \eqref{EP non stable Eq}, we see that the error $\mathcal{E}^{l,n}(t)$ can be defined as
\begin{align}\label{error definition}
\mathcal{E}^{l,n}(t)=\,&u^{l, n}(t)-u^{l, n}(0)
+\int_{0}^{t}\left[(u^{l,n}\cdot\nabla) u^{l,n}+F(u^{l,n})\right]{\rm d}t'-\int_{0}^{t}B(t',u^{l,n}){\rm d}\W.
\end{align}
Now we analyze the error as follows.

\begin{Lemma}\label{error estimate lemma}
Let $d\geq 2$ be even and $s>1+\frac{d}{2}\geq2$. For $\sigma\in\left(\frac{d}{2},\min\left\{s-1,\frac{d}{2}+1\right\}\right)$, we have that for any $T>0$ and $n\gg 1$,
\begin{align}
\E\sup_{t\in[0,T]}\|\mathcal{E}^{l,n}(t)\|^2_{H^\sigma}\leq
Cn^{-2r_s},\ \ C=C(T),\label{error estimate d even}
\end{align}
where
\begin{equation*}
r_{s}=\begin{cases}
2s-\sigma-1 \ \ \ & {\rm if} \ 1+\frac{d}{2}< s\leq 3,\\
s-\sigma+2 \ \ \ & {\rm if} \ s>3.
\end{cases}
\end{equation*}
\end{Lemma}
\begin{proof}
Direct computation shows that
\begin{equation*}
(u^{l, n}\cdot\nabla) u^{l, n}
=\left(
-2\pi l n^{-s}\sin \theta_{i}-2\pi n^{-2s+1}\sin \theta_{i}\cos \theta_{d+1-i}
\right)_{1\leq i\leq d},
\end{equation*}
which means that
\begin{align*}
u^{l, n}(t)-u^{l, n}(0)+&\int_{0}^{t}(u^{l,n}\cdot\nabla) u^{l,n}{\rm d}t'
=2\pi\int_{0}^{t}\left(-n^{-2s+1}\sin\theta_i\cos\theta_{d+1-i}\right)_{1\leq i\leq d}{\rm d}t'.
\end{align*}
Then we have
\begin{equation}
\mathcal{E}^{l,n}(t)+
\int_{0}^{t}\left[\left({2\pi}n^{-2s+1}\sin\theta_i\cos\theta_{d+1-i}\right)_{1\leq i\leq d}-F(u^{l,n})\right]{\rm d}t'
+\int_{0}^{t}B(t,u^{l,n}){\rm d}\W=0.\label{Error equation}
\end{equation}
We note that by Lemma \ref{cos sin approximate estimate},
\begin{align}
\Big\|\left(-2\pi n^{-2s+1}\sin\theta_i\cos\theta_{d+1-i}\right)_{1\leq i\leq d}\Big\|_{H^{\sigma}}
\leq C\sum^{d}_{i=1}\left\|n^{-2s+1}\sin\theta_i\cos\theta_{d+1-i}\right\|_{H^{\sigma}}
\lesssim\, & n^{-2s+1+\sigma}\lesssim n^{-r_s}.\label{error 1}
\end{align}
For $F(\cdot)=(\I-\Delta)^{-1} {\rm div}F_{1}(u)+(\I-\Delta)^{-1}F_{2}(u)$ given by \eqref{F-EP}, some calculations reveal that $F_{1}(u^{l,n})$ is a diagonal matrix such that
\begin{align*}
F_{1}(u^{l,n})=\,&{4\pi^{2}}n^{-2s+2}\times{\rm diag} (\kappa_{1},\,\cdots,\,\kappa_d), \nonumber\\
\kappa_{i}:=\,&\sin \theta_{i}(\sin \theta_{i}+\sin\theta_{d+1-i})
-\sin^{2}\theta_{d+1-i}+{\frac{1}{2}\left(\sin^{2}\theta_{1}+\cdot\cdot\cdot+\sin^{2}\theta_{d}\right)},\ \ 1\le i\le d.
\end{align*}
Therefore
\begin{align*}
{\rm div}F_{1}(u^{l, n})={8\pi^{3}}n^{-2s+3}
\big(\sin\theta_{i}\cos\theta_{d+1-i}
-\sin\theta_{d+1-i}\cos\theta_{d+1-i}\big)_{1\leq i\leq d}.
\end{align*}
Similarly, since ${\rm div}u^{l,n}=0$, we have
\begin{align*}
F_{2}(u^{l, n})
=\,&\left(-{2\pi}l n^{-s}\sin\theta_{d+1-i}-{2\pi}n^{-2s+1}\sin\theta_{d+1-i}\cos\theta_{d+1-i}\right)_{1\leq i\leq d}.
\end{align*}
Therefore
\begin{align*}
F(u^{l,n})=\left((\I-\Delta)^{-1}\Gamma_i\right)_{1\leq i\leq d},
\end{align*}
where $$\Gamma_i= 
\left({8\pi^{3}}n^{-2s+3}\sin\theta_{i}\cos\theta_{d+1-i}-
(\pi n^{-2s+1}+{4\pi^{3}}n^{-2s+3})\sin2\theta_{d+1-i}
-{2\pi}l n^{-s}\sin \theta_{d+1-i}\right).$$
Since $(\I-\Delta)^{-1}$ is bounded from $H^\sigma$ to $H^{\sigma+2}$, we can use Lemma \ref{cos sin approximate estimate} to derive that
\begin{align}
\big\|F(u^{l,n})\big\|_{H^\sigma}
\leq\,& C\sum^{d}_{i=1}\left(\left\|n^{-2s+3}\sin\theta_i\cos\theta_{d+1-i}\right\|_{H^{\sigma-2}}
+\left\|{n^{-2s+3}}\sin2\theta_{d+1-i}\right\|_{H^{\sigma-2}}\right)\nonumber\\
&+C\sum^{d}_{i=1}\left(\left\|{n^{-2s+1}}\sin2\theta_{d+1-i}\right\|_{H^{\sigma-2}}
+\|n^{-s}\sin \theta_{d+1-i}\|_{H^{\sigma-2}}\right)\nonumber\\
\lesssim\, & n^{-2s+3+\sigma-2}+n^{-2s+1+\sigma-2}+n^{-s+\sigma-2}\lesssim n^{-r_s}.\label{error 2}
\end{align}
Then we can use the It\^{o} formula to \eqref{Error equation} to find that for any $T>0$ and $t\in[0,T]$,
\begin{align*}
\E\sup_{t\in[0,T]}\|\mathcal{E}^{l,n}(t)\|^2_{H^\sigma}\leq\, &
\E\sup_{t\in[0,T]}\left|-2\int_0^{t}\IP{B(t',u^{l,n}){\rm d}\W,\mathcal{E}^{l,n}(t')}_{H^\sigma}\right|
+\sum_{i=2}^4\int_0^{T}\E|P_i|{\rm d}t,\label{error estimate}
\end{align*}
where
\begin{align*}
P_2&=-2 \IP{D^\sigma \left(n^{-2s+1}\sin\theta_i\cos\theta_{d+1-i}\right)_{1\leq i\leq d},
D^\sigma \mathcal{E}^{l,n} }_{L^2},\\
P_3&=2\bIP{D^\sigma F(u^{l,n}),D^\sigma \mathcal{E}^{l,n}}_{L^2},\
P_4=\|B(t,u^{l,n})\|_{\LL_2(\U;H^\sigma)}^2.
\end{align*}
Using \eqref{assumption G} and the BDG inequality, we find that
\begin{align}
\E\sup_{t\in[0,T]}\left|\int_0^{t}\IP{-2B(t',u^{l,n}){\rm d}\W,\mathcal{E}^{l,n}}_{H^\sigma}\right|
\leq\, & \frac12\E\sup_{t\in[0,T]}\|\mathcal{E}^{l,n}(t)\|_{H^\sigma}^2
+CTn^{-2r_s}.
\end{align}
We use \eqref{error 1} and \eqref{error 2} to find that,
\begin{align*}
\int_0^{T}\E|P_2|{\rm d}t
\leq\,& C\int_0^{T}\E\left\|(-n^{-2s+1}\sin\theta_i\cos\theta_{d+1-i})_{1\leq i\leq d}\right\|^2_{H^{\sigma}}{\rm d}t
+C\int_0^{T}\E\|\mathcal{E}^{l,n}(t)\|_{H^\sigma}^2{\rm d}t\notag\\
\leq\,& CTn^{-2r_s}+C\int_0^{T}\E\|\mathcal{E}^{l,n}(t)\|^2_{H^\sigma}{\rm d}t,
\end{align*}
\begin{align*}
\int_0^{T}\E|P_3|{\rm d}t
\leq\,& C\int_0^{T}\E\left(\|F(u^{l,n})\|_{H^{\sigma}}
\|\mathcal{E}^{l,n}(t)\|_{H^\sigma}\right){\rm d}t
\leq\, CT n^{-2r_s}
+C\int_0^{T}\E\|\mathcal{E}^{l,n}(t)\|^2_{H^\sigma}{\rm d}t,
\end{align*}
and
\begin{align*}
\int_0^{T}\E|P_4|{\rm d}t
\leq\,& C\int_0^{T}\E\|F(u^{l,n})\|^2_{H^{\sigma}}{\rm d}t\leq CT n^{-2r_s}.
\end{align*}
Collecting the above estimates into \eqref{error estimate}, we arrive at
\begin{align*}
\E\sup_{t\in[0,T]}\|\mathcal{E}^{l,n}(t)\|^2_{H^\sigma}\leq
CT n^{-2r_s}
+C\int_0^{T}\E\sup_{t'\in[0,t]}\|\mathcal{E}^{l,n}(t')\|^2_{H^\sigma}{\rm d}t.
\end{align*}
Then it follows from the Gr\"{o}nwall inequality that
\begin{align*}
\E\sup_{t\in[0,T]}\|\mathcal{E}^{l,n}(t)\|^2_{H^\sigma}\leq
Cn^{-2r_s},\ \ C=C(T),
\end{align*}
which is the desired result.
\end{proof}

\subsection{Construction of actual solutions}

Now we consider the problem \eqref{EP non stable Eq} with deterministic initial data $u^{l,n}(0,x)$, i.e.,
\begin{equation}\label{appro Cauchy problem}
\left\{\begin{aligned}
&{\rm d}u+\left[\left(u\cdot\nabla\right) u+F(u)\right]{\rm d}t
=B(t,u){\rm d}\W,\qquad t>0,\ x\in \T^d,\\
&u(0,x)=u^{l,n}(0,x), \qquad x\in \T^d,
\end{aligned} \right.
\end{equation}
where
\begin{align*}
u^{l,n}(0,x)
=\big(l n^{-1}+n^{-s}\cos n(2\pi x_{d+1-i})\big)_{1\leq i\leq d}.
\end{align*}
Then Proposition \ref{SEP non stable BdW-T} means that for each $n$, \eqref{appro Cauchy problem} has a unique maximal solution $(u_{l,n},\tau^*_{l,n})$.

\subsection{Estimates on the error}

\begin{Lemma}\label{v eta n estimate lemma}
Let $d\geq 2$ be even, $s>1+\frac{d}{2}$,
$\sigma\in\left(\frac{d}{2},\min\left\{s-1,\frac{d}{2}+1\right\}\right)$ and $r_{s}>0$ be given in Lemma \ref{error estimate lemma}. For $R\gg1$, we define
\begin{align}
\tau^R_{l,n}:=\inf\left\{t\geq0:\|u_{l,n}\|_{H^s}> R\right\},\ \ l\in\{-1,1\}.
\label{u eta n stopping time}
\end{align}
Then for any $T>0$ and $n\gg 1$, we have that for $l\in\{-1,1\}$,
\begin{align}
\E\sup_{t\in[0,T\wedge\tau^R_{l,n}]}\|u^{l,n}-u_{l,n}\|^2_{H^{\sigma}}\leq Cn^{-2r_s},\ \ C=C(R,T),\label{v sigma norm}
\end{align}
and
\begin{align}
\E\sup_{t\in[0,T\wedge\tau^R_{l,n}]}\|u^{l,n}-u_{l,n}\|^2_{H^{2s-\sigma}}\leq Cn^{2s-2\sigma},\ \ C=C(R,T).\label{v 2s-sigma norm}
\end{align}
\end{Lemma}
\begin{proof}
We first note that by Lemma \ref{cos sin approximate estimate}, for $l\in\{1,-1\}$,
\begin{equation}\label{appro solution bounded}
\|u^{l,n}(t)\|_{H^s}\lesssim 1,\ \ t\ge0,\ n\ge1,
\end{equation}
which means $\p\{\tau^R_{l,n}>0\}=1$ for any $n\ge1$ and $l\in\{-1,1\}$.
Let $v=v^{l,n}=u^{l,n}-u_{l,n}$. In view of \eqref{error definition}, \eqref{Error equation} and \eqref{appro Cauchy problem}, we see that $v$ satisfies
\begin{align*}
v(t)+&\int_{0}^{t}\left[(u^{l,n}\cdot\nabla)v+(v\cdot\nabla)u_{l,n}
+\left(-F(u_{l,n})\right)\right]{\rm d}t'\\
&=\int_{0}^{t}\left[-B(t',u_{l,n})\right]{\rm d}\W
-{2\pi}\int_{0}^{t}\left[\left(n^{-2s+1}\sin\theta_i\cos\theta_{d+1-i}\right)_{1\leq i\leq d}\right]{\rm d}t'.
\end{align*}
For any $T>0$, we use the It\^{o} formula on $[0,T\wedge\tau^R_{l,n}]$, take a supremum over $t\in[0,T\wedge\tau^R_{l,n}]$ and use the BDG inequality to find
\begin{align*}
\E\sup_{t\in[0,T\wedge\tau^R_{l,n}]}\|v\|^2_{H^{\sigma}}
\leq\,&
2\E\sup_{t\in[0,T\wedge\tau^R_{l,n}]}\left|\int_0^{t}\IP{-B(t',u_{l,n}){\rm d}\W,v}_{H^{\sigma}}\right|
+\sum_{i=2}^6\E\int_0^{T\wedge\tau^R_{l,n}}|N_i|{\rm d}t,
\end{align*}
where
\begin{align*}
N_2&=2\IP{D^\sigma \left(-{2\pi}n^{-2s+1}\sin\theta_i\cos\theta_{d+1-i}\right)_{1\leq i\leq d},
	D^\sigma v}_{L^2},\\
N_3&=-2 \bIP{D^\sigma [(v\cdot\nabla)u_{l,n}],D^\sigma v}_{L^2},\ \
N_4=-2 \bIP{D^\sigma [(u^{l,n}\cdot\nabla)v],D^\sigma v }_{L^2},\\
N_5&=2 \bIP{D^\sigma F(u_{l,n}),D^\sigma v }_{L^2},\ \
N_6=\|B(t,u_{l,n})\|_{\LL_2(\U;H^\sigma)}^2.
\end{align*}
We can first infer from Lemma \ref{F-EP Lemma} that
\begin{align*}
\|F(u_{l,n})\|_{H^\sigma}^2\lesssim\, &
\left(\|F(u^{l,n})-F(u_{l,n})\|_{H^\sigma}
+\|F(u^{l,n})\|_{H^\sigma}\right)^2
\lesssim\,  (\|u^{l,n}\|_{H^{s}}+\|u_{l,n}\|_{H^{s}})^2
\|v\|^2_{H^{\sigma}}+\|F(u^{l,n})\|^2_{H^\sigma}.
\end{align*}
From the above estimate, \eqref{assumption G}, the BDG inequality, \eqref{error 2}, \eqref{u eta n stopping time} and \eqref{appro solution bounded}, we have
\begin{align*}
&\E\sup_{t\in[0,T\wedge\tau^R_{l,n}]}\left|\int_0^{t}\IP{-2B(t',u_{l,n}){\rm d}\W,v}_{H^s}\right|\notag\\
\leq\, & 2\E
\left(\int_0^{T\wedge\tau^R_{l,n}}\|v\|_{H^\sigma}^2
\|F(u_{l,n})\|_{H^\sigma}^2{\rm d}t\right)^{\frac12}\notag\\
\leq\, & C\E \left(\sup_{t\in[0,T\wedge\tau^R_{l,n}]}\|v\|_{H^{\sigma}}^2
\int_0^{T\wedge\tau^R_{l,n}} (\|u^{l,n}\|_{H^{s}}+\|u_{l,n}\|_{H^{s}})^2
\|v\|^2_{H^{\sigma}}{\rm d}t\right)^{\frac12}\notag\\
&+C\E \left(\sup_{t\in[0,T\wedge\tau^R_{l,n}]}\|v\|_{H^{\sigma}}^2
\int_0^{T\wedge\tau^R_{l,n}}
\|F(u^{l,n})\|^2_{H^\sigma}{\rm d}t\right)^{\frac12}\notag\\
\leq\, & \frac12\E\sup_{t\in[0,T\wedge\tau^R_{l,n}]}\|v\|_{H^{\sigma}}^2
+C_R\E\int_0^{T}\sup_{t'\in[0,t\wedge\tau^R_{l,n}]}\|v(t')\|^2_{H^\sigma}{\rm d}t
+CT n^{-2r_s}.
\end{align*}
Applying Lemma \ref{F-EP Lemma}, $H^{\sigma}\hookrightarrow L^{\infty}$, integration by parts and \eqref{error 1}, we have
\begin{align*}
|N_2|\lesssim \left\|(n^{-2s+1}\sin\theta_i\cos\theta_{d+1-i})_{1\leq i\leq d}\right\|^2_{H^{\sigma}}+\|v\|^2_{H^{\sigma}}\lesssim n^{-2r_s}+\|v\|^2_{H^{\sigma}},
\end{align*}
\begin{align*}
|N_3|\lesssim\, &\|(v\cdot\nabla)u_{l,n}\|_{H^{\sigma}}\|v\|_{H^{\sigma}}
\lesssim \|v\|^2_{H^{\sigma}}\|u_{l,n}\|_{H^s},
\end{align*}
\begin{align*}
|N_5|\lesssim(\|u^{l,n}\|_{H^{s}}+\|u_{l,n}\|_{H^{s}})
\|v\|^2_{H^{\sigma}}
+\|F(u^{l,n})\|^2_{H^\sigma}+\|v\|^2_{H^\sigma},
\end{align*}
and
\begin{align*}
|N_6|\lesssim(\|u^{l,n}\|_{H^{s}}+\|u_{l,n}\|_{H^{s}})^2
\|v\|^2_{H^{\sigma}}
+\|F(u^{l,n})\|^2_{H^\sigma}.
\end{align*}
With Lemma \ref{Kato-Ponce commutator estimate} at hand, we consider the following two cases:
\begin{align*}
|N_4|\lesssim\, &\|u^{l,n}\|_{W^{\sigma,\frac{2d}{d-2}}}\|\nabla v\|_{L^d}\|v\|_{H^{\sigma}}
+\|\nabla u^{l,n}\|_{L^\infty}\|v\|^2_{H^{\sigma}}\lesssim
\|u^{l,n}\|_{H^s}\|v\|^2_{H^{\sigma}}\ \ \text{for even}\ d\geq 4,
\end{align*}
and
\begin{align*}
|N_4|\lesssim\, &\|u^{l,n}\|_{W^{\sigma,q}}\|\nabla v\|_{L^p}\|v\|_{H^{\sigma}}
+\|\nabla u^{l,n}\|_{L^\infty}\|v\|^2_{H^{\sigma}}
\ \ \text{for}\ d=2,
\end{align*}
where in the case $d=2$, $p$ will be chosen such that $\sigma-\frac{d}{2}=\sigma-1>1-\frac{2}{p}>0$ and $q$ is determined by $\frac{1}{2}=\frac{1}{q}+\frac{1}{p}$. We use
$H^s\hookrightarrow H^{\sigma+1}\hookrightarrow W^{\sigma,\frac{2d}{d-2}}$, $H^{\sigma}\hookrightarrow W^{1,d}$ for the case $d\geq4$ and use $H^{s}\hookrightarrow W^{\sigma+\frac{2}{q},q}\hookrightarrow W^{\sigma,q}$ and $H^{\sigma}\hookrightarrow W^{1,p}$ for the case $d=2$ to obtain
\begin{align*}
|N_4|\lesssim
\|u^{l,n}\|_{H^s}\|v\|^2_{H^{\sigma}}.
\end{align*}
Therefore we can infer from Lemma \ref{F-EP Lemma}, \eqref{error 2}, \eqref{u eta n stopping time} and \eqref{appro solution bounded} that
\begin{align*}
\E\int_0^{T\wedge\tau^R_{l,n}}\left(|N_2|+|N_5|+|N_6|\right){\rm d}t
\leq\, & CTn^{-2r_s}
+C_R\int_0^{T}\E
\sup_{t'\in[0,t\wedge\tau^R_{l,n}]}\|v(t')\|^2_{H^{\sigma}}{\rm d}t,
\end{align*}
and
\begin{align*}
\E\int_0^{T\wedge\tau^R_{l,n}}\left(|N_3|+|N_4|\right){\rm d}t
\leq C_R\int_0^{T}\E
\sup_{t'\in[0,t\wedge\tau^R_{l,n}]}\|v(t')\|^2_{H^{\sigma}}{\rm d}t.
\end{align*}
Over all, we arrive at
\begin{align*}
\E\sup_{t\in[0,T\wedge\tau^R_{l,n}]}\|v(t)\|^2_{H^{\sigma}}\leq CTn^{-2r_s}
+C_R\int_0^{T}\E
\sup_{t'\in[0,t\wedge\tau^R_{l,n}]}\|v(t')\|^2_{H^{\sigma}}{\rm d}t.
\end{align*}
Via the Gr\"{o}nwall inequality, we have
\begin{align*}
\E\sup_{t\in[0,T\wedge\tau^R_{l,n}]}\|v(t)\|^2_{H^{\sigma}}\leq Cn^{-2r_s},\ \ C=C(R,T),
\end{align*}
which is \eqref{v sigma norm}. For \eqref{v 2s-sigma norm}, we first note
that $u_{l,n}$ is the unique solution to \eqref{appro Cauchy problem} and $2s-\sigma>d/2+1$. For each fixed $n\ge1$, similarly, we use \eqref{u eta n stopping time} to find
\begin{align*}
\E\sup_{t\in[0,T\wedge\tau^R_{l,n}]}\|u_{l,n}(t)\|^2_{H^{2s-\sigma}}
\leq\, & 2\E\|u^{l,n}(0)\|^2_{H^{2s-\sigma}}+ C_{R,T}\int_0^{T}
\left(\E\sup_{t'\in[0,t\wedge\tau^R_{l,n}]}\|u(t')\|_{H^{2s-\sigma}}^2\right){\rm d}t.
\end{align*}
From the above estimate, we can use the Gr\"{o}nwall inequality and Lemma \ref{cos sin approximate estimate} to infer
\begin{align*}
\E\sup_{t\in[0,T\wedge\tau^R_{l,n}]}\|u_{l,n}(t)\|^2_{H^{2s-\sigma}}
\leq C\E\|u^{l,n}(0)\|^2_{H^{2s-\sigma}}
\leq C n^{2s-2\sigma},\ \ C=C(R,T).
\end{align*}
Then it follows from Lemma \ref{cos sin approximate estimate} that for some $C=C(R,T)$ and $l\in\{-1,1\}$,
\begin{align*}
\E\sup_{t\in[0,T\wedge\tau^R_{l,n}]}\|v\|^2_{H^{2s-\sigma}}
\leq C\E\sup_{t\in[0,T\wedge\tau^R_{l,n}]}\|u_{l,n}\|^2_{H^{2s-\sigma}}
+C\E\sup_{t\in[0,T\wedge\tau^R_{l,n}]}\|u^{l,n}\|^2_{H^{2s-\sigma}}
\leq C n^{2s-2\sigma},
\end{align*}
which is \eqref{v 2s-sigma norm}.
\end{proof}

\subsection{Proof for Theorem \ref{SEP non uniform}}

\begin{Lemma}\label{tau eta n lower bound}
Let $d\geq 2$ be even and $B(t,u)$ satisfy Hypothesis \ref{H-hk<F}. If for some $R_0\gg1$, the $R_0$-exiting time is strongly stable at the zero solution to \eqref{EP non stable Eq}, then for $l\in\{1,-1\}$, we have
\begin{align}
\lim_{n\rightarrow\infty} \tau^{R_0}_{l,n}=\infty\ \ \pas,\label{tau eta n lower>0}
\end{align}
where $\tau^{R_0}_{l,n}$ is given in \eqref{u eta n stopping time}.
\end{Lemma}

\begin{proof}
Since $F(0)=0$, it is clear that zero is the unique solution to \eqref{EP non stable Eq} with zero initial data under Hypothesis \ref{H-hk<F}.
Due to \eqref{appro Cauchy problem}, it follows that
\begin{align*}
\lim_{n\rightarrow\infty} \|u_{l,n}(0)-0\|_{H^{s'}}=
\lim_{n\rightarrow\infty} \|u^{l,n}(0)\|_{H^{s'}}=0 \ \ \ \forall\ s'<s.
\end{align*}
note that the $R_0$-exiting time at the zero solution is $\infty$. Therefore we see that if the $R_0$-exiting time is strongly stable at the zero solution to \eqref{EP non stable Eq}, then \eqref{tau eta n lower>0} holds true.
\end{proof}

With the above result at our disposal, now we can prove Theorem \ref{SEP non uniform}.

\begin{proof}[Proof for Theorem \ref{SEP non uniform}]
Let us first consider the case $d\geq 2$ is even. We will show that, if the $R_0$-exiting time is strongly stable at the zero solution for some $R_0\gg1$, then $(u_{-1,n},\tau_{-1,n})$ and $(u_{1,n},\tau_{1,n})$ satisfy
\ref{NU depend 1}--\ref{NU depend 4} in Theorem \ref{SEP non uniform}.

\ref{NU depend 1} For each $n>1$, for $l\in\{1,-1\}$ and for the fixed $R_0\gg1$, Lemma \ref{cos sin approximate estimate} and \eqref{u eta n stopping time} give us $\p\{\tau_{l,n}^{R_0}>0\}=1$ and Lemma \ref{tau eta n lower bound} implies \ref{NU depend 1}.

\ref{NU depend 2}
Besides, Theorem \ref{Thm-EP} and \eqref{u eta n stopping time} show that $u_{l,n}\in C([0,\tau_{l,n}];H^s)$ $\pas$ and
$$
\sup_{t\in[0,\tau_{l,n}^{R_0}]}\|u_{l,n}\|_{H^{s}}\leq R_0\ \ \pas,
$$
which gives \ref{NU depend 2}.

\ref{NU depend 3} Since $u^{-1,n}(0)$ and $u^{1,n}(0)$ are deterministic and
$$
\|u_{-1,n}(0)-u_{1,n}(0)\|_{H^{s}}=\|u^{-1,n}(0)-u^{1,n}(0)\|_{H^{s}}\lesssim n^{-1},
$$
we obtain \ref{NU depend 3} holds.

\ref{NU depend 4}
For any $T>0$, using the interpolation inequality and Lemma \ref{v eta n estimate lemma}, we see that for $l\in\{-1,1\}$ and $v=v^{l,n}=u^{l,n}-u_{l,n}$,
\begin{align*}
\left(\E\sup_{t\in[0,T\wedge\tau_{l,n}^{R_0}]}
\|v\|_{H^{s}}\right)^2
\leq\,&\left(\E\sup_{t\in[0,T\wedge\tau_{l,n}^{R_0}]}
\|v\|^2_{H^{\sigma}}\right)^{\frac12}
\left(\E\sup_{t\in[0,T\wedge\tau_{l,n}^{R_0}]}
\|v\|^2_{H^{2s-\sigma}}\right)^{\frac12}
\lesssim\,  n^{-r_{s}+(s-\sigma)}.
\end{align*}
It follows from
\begin{equation*}
0>-r_{s}+s-\sigma=\begin{cases}
1-s \ \ \ & {\rm if} \ 1+\frac{d}{2}< s\leq 3,\\
-2 \ \ \ & {\rm if} \ s>3,
\end{cases}
\end{equation*}
that for $l\in\{-1,1\}$,
\begin{equation}
\lim_{n\rightarrow\infty}
\E\sup_{t\in[0,T\wedge\tau_{l,n}^{R_0}]}\|u^{l,n}-u_{l,n}\|_{H^{s}}
=0.\label{difference tends to 0}
\end{equation}
For any given $T>0$, on account of \eqref{difference tends to 0}, Lemmas \ref{cos sin approximate estimate} and \ref{tau eta n lower bound},
we have
\begin{align*}
&\liminf_{n\rightarrow \infty}
\E\sup_{t\in[0,T\wedge\tau_{-1,n}^{R_0}\wedge\tau_{1,n}^{R_0}]}
\|u_{-1,n}(t)-u_{1,n}(t)\|_{H^s}\notag\\
\gtrsim \, &\liminf_{n\rightarrow \infty}
\E\sup_{t\in[0,T\wedge\tau_{-1,n}^{R_0}\wedge\tau_{1,n}^{R_0}]}
\|u^{-1,n}(t)-u^{1,n}(t)\|_{H^s}\notag\\
\gtrsim \, &\liminf_{n\rightarrow \infty}
\E\sup_{t\in[0,T\wedge\tau_{-1,n}^{R_0}\wedge\tau_{1,n}^{R_0}]}
\left\|n^{-s}\cos\big(2\pi n x_{d+1-i}+{2\pi}t\big)
-n^{-s}\cos\big(2\pi nx_{d+1-i}-{2\pi}t\big)\right\|_{H^s}\nonumber\\
\gtrsim \, &\liminf_{n\rightarrow \infty}
\E\sup_{t\in[0,T\wedge\tau_{-1,n}^{R_0}\wedge\tau_{1,n}^{R_0}]}
\left(n^{-s}\|\sin 2\pi n x_{d+1-i}\|_{H^s}|\sin {2\pi} t|-\|2n^{-1}\|_{H^s}\right).
\end{align*}
Using the Fatou's lemma, we arrive at
\begin{align*}
\liminf_{n\rightarrow \infty}
\E\sup_{t\in[0,T\wedge\tau_{-1,n}^{R_0}\wedge\tau_{1,n}^{R_0}]}
\|u_{-1,n}(t)-u_{1,n}(t)\|_{H^s}
\gtrsim \, & \sup_{t\in[0,T]}|\sin  {2\pi} t| ,
\end{align*}
which implies \ref{NU depend 4}.

Now we consider the case that $d\geq 3$ is odd. Instead of \eqref{approximate solutions definition 1}, we define the following divergence--free vector field as
\begin{align*}
u^{l, n}=(l n^{-1}+n^{-s}\cos\theta_{1},\, l n^{-1}+n^{-s}\cos\theta_{2},\,\cdots,\, ln^{-1}+n^{-s}\cos\theta_{d-1},\,0),
\end{align*}
where $\theta_{i}=2\pi (nx_{d-i}-l t)$ with $1\leq i\leq d-1$, $n\ge1$, $l\in\{-1,1\}$. In this case, $d-1$ is even and we can repeat the proof for Lemma \ref{error estimate lemma} to find that the error $\mathcal{E}^{l,n}(t)$ also enjoys \eqref{error estimate d even}. Moreover,
for the pathwise solutions $u_{l,n}$ to \eqref{appro Cauchy problem} with $$u_{l,n}(0)=u^{l, n}(0)=\left(l n^{-1}+n^{-s}\cos 2\pi nx_{d-i},\ 0\right)_{1\leq i\leq d-1},$$
we can basically repeat the previous procedure to show that
Lemmas \ref{v eta n estimate lemma} and \ref{tau eta n lower bound} also hold true. Therefore one can establish \ref{NU depend 1}--\ref{NU depend 4} for $u_{l,n}$ similarly.

In conclusion, we see that if for some $R_0\gg1$, the $R_0$-exiting time is strongly stable at the zero solution, then the solution map defined by \eqref{EP non stable Eq} is not uniformly continuous when $B(t,\cdot)$ satisfies Hypothesis \ref{H-hk<F}.
\end{proof}

\begin{Remark}\label{SEP non uniform remark}
From the above proof for Theorem \ref{SEP non uniform}, it is clear that if $d=1$, one can use
\begin{equation*}
u^{l, n}=l n^{-1}+n^{-s}\cos 2\pi(n x-  l t),\ n\geq 1
\end{equation*}
as a sequence of approximation solutions and repeat the other part of the proof correspondingly to obtain the similar statements in $d=1$. Therefore Theorem \ref{SEP non uniform} also holds true for $d=1$, namely the stochastic CH equation case.
\end{Remark}

\section*{Declaration} The author declare that data sharing is not applicable to this article  since no datasets were generated or analyzed during the current study.

\section*{Acknowledgement}
The author would also like to record his indebtedness to Professor Feng-Yu Wang and Professor Christian Rohde. 


\appendix


\section{Auxiliary results}\label{Section:Appendix}
In this appendix we recall and establish some auxiliary results from analysis employed in the proofs above.
We begin with introducing mollifiers. For $n\ge1$, we define the Friedrichs mollifier $J_n$ as 
\begin{align}
J_n :=\OP\big(j(\cdot/n)\big),\ \ n\ge1, \label{Define Jn}
\end{align}
where $j\in\mathscr{S}(\R^d;\R)$ (the Schwarz space of rapidly decreasing $C^\infty$ functions on $\R^d$) satisfies $0\leq j(y)\leq1$ for all $y\in \R^d$ and $j(y)=1$ for any $|y|\leq1$.  

\begin{Lemma}[\cite{Li-Liu-Tang-2021-SPA,Tang-2018-SIMA,Taylor-1991-book}]\label{Lemma-Je}
The following properties for $J_n$ hold true:
\begin{align*}
\|\I-J_n\|_{\LL(H^s;H^r)}\lesssim\,& \frac{1}{n^{s-r}},\ \ r< s,\\
\|J_n\|_{\LL(H^s;H^r)}\sim\,& O(n^{r-s}),\ \ r> s,
\end{align*}
and
\begin{align*}
[\D^s,J_n]=0,\ \ 
\IP{J_nf, g}_{L^2}=\IP{f, J_ng}_{L^2},\ \ 
\|J_n\|_{\LL(L^\infty;L^\infty)}\lesssim1,\ \
\|J_n\|_{\LL(H^s;H^s)}\leq 1,\ \ n\ge1,\ s\ge0.
\end{align*}
\end{Lemma}

\begin{Lemma}[Page 3 in \cite{Taylor-2011-PDEbook3}]\label{Te commutator} 
Let $d\geq1$ and $f,g:\K^d\rightarrow\R^d$ such that $g\in W^{1,\infty}$ and $f\in L^2$. Then for some $C=C(d)>0$,
\begin{align*}
\|[J_n, (g\cdot \nabla)]f\|_{L^2}
\leq C\|g\|_{\Wlip}\|f\|_{L^2},\ \ n\ge1.
\end{align*}
\end{Lemma}

Then we recall some estimates in Sobolev spaces $H^s$.

\begin{Lemma}[\cite{Danchin-2001-DIE,Bahouri-Chemin-Danchin-2011-book}]
\label{Lemma:product in Hs} 
For any $s>0$, and $s_1,s_2\in \R$ with $s_1+s_2>0$ and $s_1<\frac{d}{2} <s_2$, 
\begin{align*} 
&\|fg\|_{H^{s_1}}\lesssim \|f\|_{H^{s_1}}\|g\|_{H^{s_2}},\ \ f\in H^{s_1}, g\in H^{s_2}.
\end{align*}
\end{Lemma}

\begin{Lemma}[\cite{Kato-Ponce-1988-CPAM,Kenig-Ponce-Vega-1991-JAMS}]\label{Kato-Ponce commutator estimate}
If $f,g\in H^s\bigcap W^{1,\infty}$ with $s>0$, then for $p,p_i\in(1,\infty)$ with $i=2,3$ and
$\frac{1}{p}=\frac{1}{p_1}+\frac{1}{p_2}=\frac{1}{p_3}+\frac{1}{p_4}$, we have
$$\|\left[\D^s,f\I\right]g\|_{L^p}\leq C_s(\|\nabla f\|_{L^{p_1}}\|\D^{s-1}g\|_{L^{p_2}}+\|\D^sf\|_{L^{p_3}}\|g\|_{L^{p_4}}),$$
and $$\|\D^s(fg)\|_{L^p}\leq C_s(\|f\|_{L^{p_1}}\|\D^s g\|_{L^{p_2}}+\|\D^s f\|_{L^{p_3}}\|g\|_{L^{p_4}}).$$
\end{Lemma}


We also recall the following
\begin{Lemma}[\cite{Tang-Liu-2015-ZAMP,Zhao-Yang-Li-2018-JMAA}]\label{cos sin approximate estimate}
Let $\sigma, \alpha\in\R$. If $n\gg 1$, then
\begin{align*}
\|\sin(n (2\pi x)-\alpha)\|_{H^\sigma(\T;\R)}, \ 
\|\cos(n (2\pi x)-\alpha)\|_{H^\sigma(\T;\R)},\ 
\|\cos(n (2\pi x)-\alpha)\sin(n (2\pi y)-\alpha)\|_{H^\sigma(\T^2;\R)}
\approx n^{\sigma}. 
\end{align*}
\end{Lemma}

The following lemmas with single $\P$ and a pair $(\P_1,\P_2)$ on $\R^d$ are well known in the literature, and they can be easily extended to the case on $\T^d$ (cf. see for example \cite[Theorem 4.5.3, Corollaries 4.5.7 and 4.6.13]{Ruzhansky-Turunen-2010-Book} and  \cite{McLean-1991-MN}). 

\begin{Lemma}[\cite{Taylor-1974-note,Taylor-1991-book,Benzoni-Gavage-Serre-2007-Book}]\label{LOP} Let $\P\in \OP\SS^{s}$, $s\in \R$. Then $\P^*\in\OP\SS^{s}$, and $\P\in \LL(H^q;H^{q-s})$ for any $q\in\R$.  For any $\P_i\in\OP\SS^{r_i}, r_i\in \R, i=1,2$, 
$$\P_1\P_2\in \OP\SS^{r_1+r_2}.$$
Moreover, if their symbols are commuting matrices, then
$$[\P_1,\P_2]\in \OP\SS^{r_1+r_2-1}.$$ \end{Lemma} 


\begin{Lemma}[Proposition 4.2 in \cite{Taylor-2003-PAMS}]\label{commutator:Taylor 2}
Let $\P \in \OP\S^{r} $ with $r \geq 0$. For any $\sigma>1+\frac{d}{2}$ and $q\in [0, \sigma-r]$, 
$$
\left\|\left[\P, g\I\right] u\right\|_{H^{q}} \lesssim \|g\|_{H^{\sigma}}\|u\|_{H^{q+r-1}},\ \ g\in H^{\sigma},\ u\in H^{q+r-1}.
$$

\end{Lemma}

\begin{Lemma}\label{Lemma:(pn qm) [P Q]}
Let $r_1,r_2\in\R$. Assume that $\{(\mathfrak{p}_n,\mathfrak{q}_k)\}_{n,k\ge1}\subset \SS^{r_1}\times \SS^{r_2}$ is bounded. If for all $n,k\ge1$, $\mathfrak{p}_n$ and $\mathfrak{q}_k$ are commuting matrices,
then 
\begin{equation*} 
\sup_{n,k}\|\left[{\rm OP}(\mathfrak{p}_n),{\rm OP}(\mathfrak{q}_k)\right]\|_{\LL(H^{r_1+r_2-1};L^2)}<\infty.
\end{equation*}
\end{Lemma}

\begin{proof}
Since $\SS^{r_1}\times \SS^{r_2}$ is a Fr\'{e}chet space, it suffices to show that the mapping $(\mathfrak{p},\mathfrak{q})\to \left[{\rm OP}(\mathfrak{p}),{\rm OP}(\mathfrak{q})\right]$ is bilinear and continuous from $\SS^{r_1}\times \SS^{r_2}$ to $\LL(H^{r_1+r_2-1};L^2)$. Bilinearity is obvious, and now we prove the continuity. To this end, we denote by $\mathfrak{p}_1\# \mathfrak{p}_2$  the symbol of the operator product ${\rm OP}(\mathfrak{p}_1){\rm OP}(\mathfrak{p}_2)$, i.e., ${\rm OP}(\mathfrak{p}_1){\rm OP}(\mathfrak{p}_2)={\rm OP}(\mathfrak{p}_1\#\mathfrak{p}_2)$. 

On one hand, for two symbols $\mathfrak{p}_1$ and $\mathfrak{p}_2$, it is well-known that the mapping $\SS^{r_1}\times \SS^{r_2}\ni (\mathfrak{p}_1,\mathfrak{p}_2)\mapsto \mathfrak{p}_1\# \mathfrak{p}_2\in \SS^{r_1+r_2}$ is continuous (cf. \cite[Theorem 1.2.16]{Nicola-Rodino-2010-book}, \cite[Page 72]{Abels-2012-Book}).

On the other hand, when $\mathfrak{p}_1$ and $\mathfrak{p}_2$ are commuting matrices,  some direct computations (cf. \cite[Corollary 4.1]{Alinhac-Gerard-2007-Book} or \cite[Theorem C.3]{Benzoni-Gavage-Serre-2007-Book}) yield
$\mathfrak{p}_1\# \mathfrak{p}_2-\mathfrak{p}_2\# \mathfrak{p}_1\in \SS^{r_1+r_2-1}$. 

Therefore, the mapping
\begin{align*}
\textbf{T}:\SS^{r_1}\times \SS^{r_2}\ni (\mathfrak{p}_1,\mathfrak{p}_2)\mapsto \textbf{T}(\mathfrak{p}_1,\mathfrak{p}_2):=\mathfrak{p}_1\# \mathfrak{p}_2- \mathfrak{p}_2\# \mathfrak{p}_1\in \SS^{r_1+r_2-1}
\end{align*} 
is continuous, which together with \eqref{OP continuous} implies that
$(\mathfrak{p},\mathfrak{q})\to \left[{\rm OP}(\mathfrak{p}),{\rm OP}(\mathfrak{q})\right]={\rm OP}(\textbf{T}(\mathfrak{p},\mathfrak{q}))$ is continuous from $\SS^{r_1}\times \SS^{r_2}$ to $\LL(H^{r_1+r_2-1};L^2)$. 
\end{proof}


\begin{Lemma}\label{commutator:Taylor 2-n}
Let $r \geq 0$, $\sigma>d/2+1$ and $q \in[0,\sigma-r]$. 
If $\{\mathfrak{p}_n\}_{n\ge1}\subset\S^{r}$ is bounded and $g\in H^{\sigma}$, then there is a constant $C>0$ independent of $n$ such that
\begin{equation*}
\|[{\rm OP}(\mathfrak{p}_n),g\I] u\|_{H^{q}} \leq C\|g\|_{H^{\sigma}}\|u\|_{H^{q+r-1}},\ \ u\in H^{q+r-1}.
\end{equation*} 
\end{Lemma}
\begin{proof}

To begin with, we define a sequence of linear operator $\mathfrak{T}:\OP\S^{s}\ni \P\mapsto [\P,g\I]$. Then Lemma \ref{commutator:Taylor 2} shows
\begin{equation*}
\|[\mathfrak{T}(\P)]u\|_{H^q}\leq
C\|g\|_{H^{\sigma}}\|u\|_{H^{q+r-1}},\ \ C=C(\P),
\end{equation*}
which means 
\begin{equation*}
\|\mathfrak{T}(\P)\|_{\LL(H^{q+r-1};H^{q})}
=\sup_{\|u\|_{H^{q+r-1}}\ne 0}\frac{\|[\mathfrak{T}(\P)]u\|_{H^q}}{\|u\|_{H^{q+r-1}}}
\leq C \|g\|_{H^{\sigma}}, \ \ C=C(\P).
\end{equation*}
The above estimate and \eqref{OP continuous} imply that under the above assumptions for $r,\sigma$ and $q$,
\begin{equation*}
\mathfrak{T}{\rm OP}:\S^r\to 
\LL(H^{q+r-1};H^{q})\ \text{is continuous},
\end{equation*}
which implies the desired estimate.
\end{proof}

%
%
%

\end{document}